\documentclass[10pt]{article}
\usepackage{makeidx}
\makeindex
\usepackage{amsmath}
\usepackage{amssymb}
\usepackage{amsthm}
\usepackage[top=2.5cm, bottom=3cm, left=2.65cm, right=2.5cm]{geometry}

\begin{document}

\def\ds{\displaystyle}

\newtheorem{theorem}{Theorem}
\newtheorem{lemma}{Lemma}
\newtheorem{corollary}{Corollary}
\newtheorem{remark}{Remark}
\newtheorem{example}{Example}

\title{Reaction-Diffusion Problems on Time-Periodic Domains}
\author{Jane Allwright}
\date{}
\maketitle

\section*{Abstract}
Reaction-diffusion equations are studied on bounded, time-periodic domains with zero Dirichlet boundary conditions.
The long-time behaviour is shown to depend on the principal periodic eigenvalue of a transformed periodic-parabolic problem. We prove upper and lower bounds on this eigenvalue under a range of different assumptions on the domain, and apply them to examples.
The principal eigenvalue is considered as a function of the frequency, and results are given regarding its behaviour in the small and large frequency limits. A monotonicity property with respect to frequency is also proven.
A reaction-diffusion problem with a class of monostable nonlinearity is then studied on a periodic domain, and we prove convergence to either zero or a unique positive periodic solution.

\section{Introduction}
In this paper we study reaction-diffusion equations on time-dependent domains $\Omega(t)\subset \mathbb{R}^N$ which are bounded, connected, and periodic with period $T$: $\Omega(t) \equiv \Omega(t+T)$. 
Specifically, we consider non-negative solutions $\psi(x,t)$ to
\begin{equation}\label{eq_psi}
\frac{\partial \psi}{\partial t} = D \nabla^2 \psi +f(\psi)  \qquad \textrm{for }x\in \Omega(t)
\end{equation}
\begin{equation}\label{eq_psi_BC}
\psi(x,t)=0 \qquad\textrm{for }x\in \partial\Omega(t) 
\end{equation}
where $f$ is either linear: $f(\psi)=f'(0)\psi$, or a nonlinear, Lipschitz continuous function which satisfies the following conditions for some $K>0$:
\begin{align}\label{assumptions_on_f}
& f(0)=f(K)=0, \quad f'(0)>0\textrm{ exists}, \quad \frac{f(k)}{k}\textrm{ is non-increasing on }k>0.
\end{align}
Note that these assumptions on $f$ imply that $f(k)\leq f'(0)k$ for $k> 0$.

Such reaction-diffusion problems, with $\psi \geq 0$, can be used to model population dynamics, chemical diffusion, and other biological, ecological and physical processes. We are interested in domains whose boundaries change periodically due to some external influence.
In the context of population dynamics, this could represent a habitat whose boundary moves periodically due to the seasonal variation in temperature, water level, or snow cover, or due to a periodic cycle of land usage.
Population models typically include Dirichlet, Neumann or Robin boundary conditions. In this study we shall impose zero Dirichlet conditions, since these lead to the more interesting mathematical results and comparison with the solution on a fixed domain.

For non-negative solutions to the problem \eqref{eq_psi}, \eqref{eq_psi_BC} but on a constant bounded domain $\Omega$, the long-time behaviour depends on $\lambda(\Omega)$, where this is the principal Dirichlet eigenvalue of $-\nabla^2$ on $\Omega$.
If $f'(0)<D\lambda(\Omega)$ then the solution tends to zero as $t\rightarrow\infty$. 
If $f'(0)>D\lambda(\Omega)$ then the solution to the linear equation tends to infinity, and the solution to the nonlinear problem has a nontrivial lower bound. 
(This is straightforward to show if we bound the initial conditions above and below by appropriately large and small multiples of the principal eigenfunction.)
For an interval $0<x<L_0$ the critical threshold value for $f'(0)$ is $\frac{D\pi^2}{L_0^2}$, which leads to the definition of `critical length' as
\begin{equation}
L_{crit}=\pi\sqrt{\frac{D}{f'(0)}}.
\end{equation}
Thus, on a fixed interval $0<x<L_0$, there is a positive stationary solution to the linear problem if $L_0=L_{crit}$; otherwise the solution tends to zero or infinity depending whether $L_0<L_{crit}$ or $L_0>L_{crit}$ respectively.
Such results only apply for constant domains, not for those that vary with time.

There is very little published work addressing similar matters on time-dependent domains; however see \cite{JA1} and \cite{JA2} which consider the linear version of \eqref{eq_psi}, \eqref{eq_psi_BC} on time-dependent intervals $A(t)<x<A(t)+L(t)$.
The approach in both \cite{JA1} and \cite{JA2} is based upon a change of variables onto a fixed domain, followed by other changes of variables which allow the construction of exact solutions, subsolutions and supersolutions.
Here, similarly, we start in Section \ref{section_ppe} by
transforming the problem onto a fixed domain. Now, due to the assumed $T$-periodicity of the original domain, this converts the problem to a periodic-parabolic equation.
Although periodic-parabolic problems have been studied by several authors, and we make use of results from Castro and Lazer \cite{CasLaz}, Hess \cite{Hess}, Peng and Zhao \cite{PengZhao}, and Liu, Lou, Peng and Zhou \cite{LiuLouPengZho}, none of these have worked specifically on time-periodic domains or the periodic-parabolic equations that arise from them.

The results of Castro and Lazer \cite{CasLaz} show that the long-time behaviour of the solution to our transformed problem is determined by a principal periodic eigenvalue, $\mu$. It is worth emphasising the importance of \cite[Theorem 1]{CasLaz} (the existence and uniqueness of the principal periodic eigenvalue and eigenfunction) for our problem on a time-periodic domain. The principal periodic eigenvalue $\mu$ is an important threshold value, and its existence distinguishes the study of time-periodic domains from that of general time-dependent domains. Various results concerning the long-time behaviour of the solution can now be stated in terms of bounds on this eigenvalue.

In Section \ref{section_bounds} we derive upper and lower bounds on the principal periodic eigenvalue $\mu$ associated with a time-periodic domain $\Omega(t)$, under a range of different assumptions on $\Omega(t)$. These bounds, and their derivations, give original and useful insight into how problems on time-periodic domains behave. We also apply these bounds to some illustrative examples, including the interval $0<x<L(t)$ with
\begin{equation}\label{eq_intro_example}
L(t)=L_0(1+\varepsilon\sin(\omega t)),
\end{equation}
with $\omega>0$ and $0<\varepsilon<1$. In this case (see Example \ref{example_periodic_L_om_eps}), the bounds on $\mu$ imply that if
\begin{equation}
f'(0)<\frac{D\pi^2}{L_0^2(1-\varepsilon^2)^{3/2}},
\end{equation}
that is, if
\begin{equation}
L_0 <\pi\sqrt{\frac{D}{f'(0)(1-\varepsilon^2)^{3/2}}} = \frac{L_{crit}}{(1-\varepsilon^2)^{3/4}},
\end{equation}
then the solution to \eqref{eq_psi}, \eqref{eq_psi_BC} tends to zero. This result holds independently of the frequency $\omega>0$. Results from Section \ref{section_mu_omega} also imply that if we have the opposite inequality,
\begin{equation}
L_0 > \frac{L_{crit}}{(1-\varepsilon^2)^{3/4}},
\end{equation}
then there exist $\omega>0$ such that the solution (to the linear version of \eqref{eq_psi}, \eqref{eq_psi_BC}) tends to infinity.

Such results demonstrate that the `critical length' has a more intricate role for time-dependent domains than for constant domains.
This is also shown in \cite{JA2}, where the author studies the problem on an interval $0<x<L(t)$ of general time-dependent length $L(t)$. Even if $L(t)<L_{crit}$ for all $t$, the solution may or may not converge to zero. In \cite{JA2}, conditions on $L(t)$ are derived that guarantee each outcome. In particular, if 
\begin{equation}
L(t)=L_{crit}(1-\varepsilon (t+1)^{-k})
\end{equation}
with $0<\varepsilon<1$ and $0<k\leq 1$, then $\psi(x,t)\rightarrow 0$ uniformly in $x$ as $t\rightarrow\infty$; however if $k>1$ then there is a non-trivial lower bound: $\psi(x,t)\geq B\sin\left(\frac{\pi x}{L(t)}\right)$ for some $B>0$.

Here, the case \eqref{eq_intro_example} is just one example that we consider; we also study more general time-periodic intervals $A(t)<x<A(t)+L(t)$. Moreover, several of our results are valid for time-periodic domains $\Omega(t)$ in $\mathbb{R}^N$.

In Section \ref{section_mu_omega} we consider $\mu$ as a function of the frequency $\omega=\frac{2\pi}{T}$ and prove results concerning the small and large frequency limits, as well as a monotonicity property.
The first step is to convert our principal eigenvalue problem into a $1$-periodic problem; it then becomes an equation of the form
\begin{equation}
\frac{\omega}{2\pi}\frac{\partial \phi}{\partial s}- \mathcal{L}_{\omega}(\xi,s)\phi = \mu(\omega)\phi(\xi,s) \qquad \xi\in\Omega_0, s\in[0,1].
\end{equation}
The operator $\mathcal{L}_{\omega}(\xi,s)$ is $1$-periodic in $s$, but it has certain terms that depend on, and scale with, $\omega$.

Our proofs of the limit $\lim_{\omega\rightarrow0} \mu(\omega)$, and of a monotonicity result, are inspired by methods from \cite{LiuLouPengZho} but have to be adapted to our own case.
In \cite{LiuLouPengZho}, Liu, Lou, Peng and Zhou study the dependence of a principal periodic eigenvalue on the frequency, but for a different problem.
They consider the principal periodic eigenvalue $\hat{\lambda}(\omega)$ of
\begin{equation}
\frac{\omega}{2\pi}\frac{\partial \hat{\phi}}{\partial s}- \mathcal{\hat{L}}(\xi,s)\hat{\phi} = \hat{\lambda}(\omega)\hat{\phi}(\xi,s) \qquad \xi\in\Omega_0, s\in[0,1]
\end{equation}
where the coefficients of the operator $\mathcal{\hat{L}}(\xi,s)$ are $1$-periodic in $s$ and do not depend on $\omega$.
The result of \cite[Theorem 1.3(i)]{LiuLouPengZho} is that
\begin{equation}\label{eq_lambda_om0}
\lim_{\omega\rightarrow 0}  \hat{\lambda}(\omega) = \int_0^1 \lambda^{0}(s)ds
\end{equation}
where for each $0\leq s\leq 1$, $\lambda^{0}(s)$ is the principal Dirichlet eigenvalue of the elliptic operator $-\mathcal{\hat{L}}(\xi,s)$ on $\Omega_0$.
This result can be extended in a natural way to the operator $-\mathcal{L}_{\omega}(\xi,s)$, since the $\omega$-dependence is sufficiently smooth and the uniform ellipticity condition holds independently of $\omega$ as $\omega\rightarrow 0$. Thus, by adapting the proof of \cite[Theorem 1.3(i)]{LiuLouPengZho}, in Theorem \ref{theorem_om_tends_0} we identify $\lim_{\omega\rightarrow0} \mu(\omega)$ for a $\frac{2\pi}{\omega}$-periodic domain in any dimension. In Corollary \ref{corollary_omega_tends_0} we then show that if $\mu(\omega)$ is the principal periodic eigenvalue associated with $\Omega(t)=\tilde{\Omega}\left(\frac{\omega t}{2\pi}\right)$, where $\tilde{\Omega}(s)$ is a bounded and $1$-periodic domain, then
\begin{equation}
\lim_{\omega\rightarrow 0}\mu(\omega) =\int_0^1 D\lambda(\tilde{\Omega}(s))ds.
\end{equation}
In this formula, for each $0\leq s\leq 1$, $\lambda (\tilde{\Omega}(s))$ denotes the principal Dirichlet eigenvalue of $-\nabla^2$ on $\tilde{\Omega}(s)$.

Next we consider $\omega\rightarrow\infty$. The large frequency limit of $\hat{\lambda}(\omega)$ (the principal periodic eigenvalue of Liu, Lou, Peng and Zhou in \cite{LiuLouPengZho}) is given in \cite[Theorem 1.3(ii)]{LiuLouPengZho}. By adapting an argument from \cite[Theorem 3.10]{Nad}, they prove that $\lim_{\omega\rightarrow \infty} \hat{\lambda}(\omega) = \lambda_{\infty}$ where $\lambda_{\infty}$ is the principal eigenvalue of the elliptic operator whose coefficients are equal to the time-averages of those of $-\mathcal{\hat{L}}(\xi,s)$. However, neither this result nor the analysis of \cite{Nad} applies in cases such as ours, when some of the coefficients depend on $\omega$ and become unbounded as $\omega\rightarrow\infty$.

In our case, as we shall see, the behaviour of $\mu(\omega)$ as $\omega\rightarrow \infty$ depends on the detail of the problem. 
We show that very different types of asymptotic behaviour of $\mu(\omega)$ are possible as $\omega \rightarrow\infty$. Indeed, let
\begin{equation}
L(t)=L_0 l\left(\frac{\omega t}{2\pi}\right), \qquad A(t)=A_0 a\left(\frac{\omega t}{2\pi}\right)
\end{equation}
where $L_0>0$, $A_0\geq0$, $\omega>0$, and where
$l(\cdot)$ and $a(\cdot)$ are $1$-periodic functions with $l\geq 1$. Let $\mu(\omega)$ be the principal periodic eigenvalue associated with the domain $(A(t),A(t)+L(t))$. We show that if $a(\cdot)$ is constant then $\mu(\omega)=O(1)$ as $\omega \rightarrow\infty$, but if $a(\cdot)$ is non-constant then there exist $C_1$, $C_2$ such that: if $\frac{A_0}{L_0}<C_1$ then $\mu(\omega)=O(1)$ as $\omega \rightarrow\infty$, and if
$\frac{A_0}{L_0}>C_2$ then $\mu(\omega) \rightarrow\infty$ at the rate $\omega^2$ as $\omega \rightarrow\infty$ (see Theorem \ref{theorem_om_tends_infty}).

In \cite[Theorem 1.1]{LiuLouPengZho}, Liu, Lou, Peng and Zhou prove that if their parabolic operator has no advection term then the principal periodic eigenvalue $\hat{\lambda}(\omega)$ is non-decreasing with respect to $\omega>0$.
In the final part of Section \ref{section_mu_omega} we similarly show that the principal periodic eigenvalue $\mu(\omega)$ associated with a periodic domain $\Omega(t) \subset \mathbb{R}^N$ is also non-decreasing with respect to $\omega>0$.

We conclude the paper with Section \ref{section_nonlinear}, in which we consider the nonlinear problem on a periodic domain in $\mathbb{R}^N$, with $f$ satisfying \eqref{assumptions_on_f}. Using a result of Hess \cite{Hess} and methods involving the Poincar\'{e} map, we prove convergence to zero if $f'(0)<\mu$, or to a unique positive periodic solution if $f'(0)>\mu$. That is, if $f'(0)>\mu$ then there is a unique positive solution $\psi^{*}(x,t)$ to \eqref{eq_psi}, \eqref{eq_psi_BC} such that $\psi^{*}(x,t)\equiv \psi^{*}(x,t+T)$, and for any non-negative initial conditions, the solution $\psi(x,nT+t)$ converges uniformly to $\psi^{*}(x,t)$ as $n\rightarrow\infty$.

\section{Principal periodic eigenvalue $\mu$}\label{section_ppe}
Throughout the paper, we shall assume there is a one-to-one mapping $h(\cdot ,t): \overline{\Omega(t)}\rightarrow \overline{\Omega_0}$ which transforms $\Omega(t)$ into a bounded, connected reference domain $\Omega_0$ with sufficiently smooth boundary (at least $C^{2+\varepsilon}$ for some $\varepsilon>0$), and such that the change of variables $\xi=h(x,t)$ and $u(\xi,t)=\psi(x,t)$ transforms \eqref{eq_psi}, \eqref{eq_psi_BC} into a parabolic equation of the form
\begin{equation}\label{eq_u_ij}
\frac{\partial u}{\partial t} = \mathcal{L}(\xi,t) u + f(u) \qquad\textrm{for } \xi\in\Omega_0
\end{equation}
\begin{equation}\label{eq_u_ij_BC}
u(\xi,t)=0\qquad\textrm{for } \xi\in\partial\Omega_0,
\end{equation}
where
\begin{equation}\label{eq_L_aij_bj_cj}
\mathcal{L}(\xi,t) u  = \sum_{i,j}a_{ij}(\xi,t)\frac{\partial^2 u}{\partial \xi_i\partial \xi_j} +\sum_j \left(b_j(\xi,t)+c_j(\xi,t)\right)\frac{\partial u}{\partial \xi_j} \qquad\textrm{for } \xi\in \Omega_0,
\end{equation}
\begin{equation}\label{eq_aij_bj_cj}
a_{ij}(\xi,t)= \sum_{k}D\left(\frac{\partial h_i}{\partial x_k}\frac{\partial h_j}{\partial x_k}\right), \qquad
b_j(\xi,t)= -\frac{\partial h_j}{\partial t}, \qquad
c_j(\xi,t)=D\nabla^2 h_j.
\end{equation}
We assume that the map $h$ is such that $a_{ij}$, $b_j$, $c_j$ belong to $C^{\alpha,\alpha/2}(\overline{\Omega_0}\times[0,T])$ for some $\alpha>0$, and that $a_{ij}$ is uniformly elliptic.
For example, if the time-dependent domain is an interval $A(t)<x<A(t)+L(t)$ where $L(t)>0$ and $A(t)$ are $C^{2+\alpha}$ functions for some $\alpha>0$, then we can transform onto a fixed reference domain $0<\xi< L_0$ by letting $\xi=\left(\frac{x-A(t)}{L(t)}\right)L_0$. The solution $u(\xi,t)=\psi(x,t)$ then satisfies
\begin{equation}\label{eq_u_xi}
\frac{\partial u}{\partial t} = D \frac{L_0^2}{L(t)^2} \frac{\partial^2 u}{\partial \xi^2}+\left(\frac{\dot{A}(t)L_0+\xi\dot{L}(t)}{L(t)}\right)\frac{\partial u}{\partial \xi} +f(u)  \qquad \textrm{in } 0< \xi< L_0
\end{equation}
with $u(\xi,t)=0$ at $\xi=0$ and $\xi=L_0$.

Since $\Omega(t)$ is periodic with period $T$, the map $h$ and the coefficients of $\mathcal{L}$ are also $T$-periodic in $t$.
By Theorem 1 of Castro and Lazer \cite{CasLaz} there exists a value $\mu$ and a function $\phi(\xi,t)$ such that
\begin{equation}\label{eq_princ_eig1}
\frac{\partial \phi}{\partial t} - \mathcal{L} \phi = \mu\phi  \qquad \textrm{in } \xi\in\Omega_0,\ t\in\mathbb{R}
\end{equation}
\begin{equation}\label{eq_princ_eig2}
\phi(\xi,t)=0 \qquad \textrm\qquad\textrm{for } \xi \in\partial\Omega_0
\end{equation}
\begin{equation}\label{eq_princ_eig3}
\phi(\xi,t)>0 \qquad \textrm\qquad\textrm{for } \xi\in\Omega_0
\end{equation}
\begin{equation}\label{eq_princ_eig4}
\phi(\xi,t)\equiv\phi(\xi,t+T) .
\end{equation}
This function $\phi$ is unique up to scaling \cite[Theorem 1]{CasLaz}, and is called the principal periodic eigenfunction, while $\mu$ is called the principal periodic eigenvalue. Here, we shall say that $\mu$ is `the principal periodic eigenvalue of $\Omega(t)$' to mean that it is the principal periodic eigenvalue of \eqref{eq_princ_eig1}, \eqref{eq_princ_eig2}, \eqref{eq_princ_eig3}, \eqref{eq_princ_eig4}, when $\mathcal{L}$ is defined by \eqref{eq_L_aij_bj_cj}, \eqref{eq_aij_bj_cj}.

The function $u(\xi,t)=\phi(x,t)e^{(f'(0)-\mu)t}$ is a solution to the linear reaction-diffusion equation \eqref{eq_u_ij}, \eqref{eq_u_ij_BC} with $f(u)=f'(0) u$. So, if the initial conditions satisfy $b\phi(\xi,0)\leq u(\xi,0)\leq a\phi(\xi,0)$ for some $0<b\leq a$, then by the comparison principle the solution to the linear equation satisfies
\begin{equation}\label{eq_u_phi_mu}
b\phi(\xi,t)e^{(f'(0)-\mu)t}\leq u(\xi,t)\leq a\phi(\xi,t)e^{(f'(0)-\mu)t} \qquad \textrm{for all }t\geq 0.
\end{equation}
The principal periodic eigenvalue is therefore a threshold such that
if $f'(0)>\mu$ then $u(\xi,t)\rightarrow\infty$ as $t\rightarrow\infty$, whereas if $f'(0)<\mu$ then $u(\xi,t)\rightarrow 0$. See also \cite[page 192]{CantrellCosner}.

\section{Bounds on $\mu$ and examples}\label{section_bounds}
In this section we prove some bounds on $\mu$, the principal periodic eigenvalue of a $T$-periodic domain $\Omega(t)$.
First we give a lower bound, which is valid for time-dependent domains in any dimension.
\begin{theorem}\label{theorem_mu_bounds1a}
Let $\Omega(t)$ be $T$-periodic and let $\mu$ be the principal periodic eigenvalue of $\Omega(t)$.
At each fixed time $0\leq t\leq T$, let $\lambda(\Omega(t))$ be the principal eigenvalue of $-\nabla^2$ on the domain $\Omega(t)$ with zero Dirichlet boundary conditions. Then
\begin{equation}\label{eq_mu_int_Dlambda}
\mu\geq \frac{1}{T}\int_0^T D\lambda(\Omega(t)) dt.
\end{equation}
\end{theorem}
\begin{proof}
Let $\psi$ be a solution to the problem \eqref{eq_psi}, \eqref{eq_psi_BC} with $f(\psi)=f'(0) \psi$, and define
\begin{equation}
E(t)=\frac{1}{2}\int_{\Omega(t)}\psi(x,t)^2 dx.
\end{equation}
We differentiate $E(t)$, noting that the additional contribution from the moving boundary $\partial\Omega(t)$ vanishes due to the zero Dirichlet conditions. It follows from \eqref{eq_psi}, \eqref{eq_psi_BC} that
\begin{equation}
\frac{dE}{dt}=\int_{\Omega(t)}\psi(D\nabla^2\psi +f'(0)\psi )dx = \int_{\Omega(t)}(-D\vert\nabla\psi\vert^2 +f'(0)\psi^2)dx.
\end{equation}
Then for each time $t$, we use Poincar\'{e}'s inequality to get
\begin{align}
\frac{dE}{dt}
&\leq \left(-D\lambda(\Omega(t)) +f'(0)\right) \int_{\Omega(t)} \psi^2 dx =2\left(f'(0)-D\lambda(\Omega(t))\right)E(t).
\end{align}
Therefore, for $t\geq 0$
\begin{align}
0\leq E(t) &\leq E(0)\exp{\left(2\int_0^t \left(f'(0)-D\lambda(\Omega(\zeta)) \right) d\zeta \right)}\\
&= E(0)\exp{\left(2\left(f'(0)t - \frac{t}{T}\int_0^T D \lambda(\Omega(\zeta))d\zeta \right)+O(1) \right)} \quad \textrm{as } t\rightarrow\infty\label{eq_Et_exp_decay}
\end{align}
where in the last line we have used the fact that $\lambda(\Omega(t))$ is $T$-periodic.
So, if $f'(0)-\frac{1}{T}\int_0^T D \lambda(\Omega(\zeta))d\zeta<0$ then \eqref{eq_Et_exp_decay} implies that $\int_{\Omega(t)}\psi(x,t)^2 dx \rightarrow 0$ as $t\rightarrow\infty$.
Therefore, the condition $f'(0)-\frac{1}{T}\int_0^T D \lambda(\zeta)d\zeta<0$ must imply that $f'(0)<\mu$, and so the bound \eqref{eq_mu_int_Dlambda} is proved.
\end{proof}
\begin{remark}
If we were instead interested in the principal periodic eigenvalue $\mu_V$ of an operator
$\frac{\partial}{\partial t}-D\nabla^2+V(x,t)$ on $\Omega(t)$, where $V$ was a continuous function on $\Omega(t)$ and periodic in $t$ with the same period as the domain, then the same proof would show that
\begin{equation}\nonumber
\mu\geq \frac{1}{T}\int_0^T \left( D\lambda(\Omega(t))-\max_y V(y,t) \right) dt.
\end{equation}
\end{remark}
\begin{example}\label{example_periodic_AL1}
Let $\Omega(t)=(A(t), A(t)+L(t))$ where $L(t)>0$ and $A(t)$ are $T$-periodic. By Theorem \ref{theorem_mu_bounds1a},
\begin{equation}
\mu\geq \frac{1}{T}\int_0^T \frac{D\pi^2}{L(t)^2} dt.
\end{equation}
In particular, if $L(t)\equiv l>0$ is constant and $A(t)$ is periodic then we have the lower bound $\mu \geq \frac{D\pi^2}{l^2}$. This means that whenever the solution on the fixed interval $(0,l)$ tends to zero (i.e. $f'(0)<\frac{D\pi^2}{l^2}$), the solution on a periodic interval $(A(t),A(t)+l)$ also tends to zero (i.e. $f'(0)<\mu$) for every periodic function $A(t)$.
\end{example}

The next result follows from the comparison principle, and is also valid for time-dependent domains in any dimension.
\begin{theorem}\label{theorem_mu_bounds1b}
Let $\Omega(t)$ be $T$-periodic and let $\mu$ be the principal periodic eigenvalue of $\Omega(t)$.
Suppose there is a domain $\Omega_1$ such that $\Omega_1 \subset \Omega(t)$ for all $t$, and let $\lambda(\Omega_1)$ be the principal eigenvalue of $-\nabla^2$ on $\Omega_1$ with zero Dirichlet boundary conditions. Then
\begin{equation}\label{eq_mu_Omega1}
\mu\leq D\lambda(\Omega_1).
\end{equation}
\end{theorem}
\begin{proof}
Let $\psi_1$ be the solution to \eqref{eq_psi}, \eqref{eq_psi_BC} with $f(u)=f'(0) u$, but on the fixed domain $\Omega_1$, and with non-trivial initial conditions satisfying $0\leq \psi_1(x,0)\leq \psi(x,0)$. By the comparison principle, $0\leq \psi_1(x,t)\leq \psi(x,t)$ for all $x\in \Omega_1$ and $t\geq0$. If $f'(0)<\mu$ then $\psi\rightarrow 0$ as $t\rightarrow\infty$, which implies that $\psi_1\rightarrow 0$ and so that $f'(0)<D\lambda(\Omega_1)$. Therefore we conclude that $\mu \leq D\lambda(\Omega_1)$.
\end{proof}
\begin{example}\label{example_periodic_AL2}
Let $\Omega(t)=(A(t), A(t)+L(t))$ for some $T$-periodic functions $L(t)>0$ and $A(t)$ satisfying
\begin{equation}
\max_{[0,T]} A < \min_{[0,T]}(A+L).
\end{equation}
The fixed interval $\Omega_1:=(\max(A), \min(A+L))$ is always contained within the domain $(A(t), A(t)+L(t))$. By Theorem \ref{theorem_mu_bounds1b}, we get the upper bound
\begin{equation}
\mu \leq \frac{D\pi^2}{(\min(A+L) - \max A)^2}.
\end{equation}
In particular, if $A(t)\equiv A(0)$ is constant and $L(t)>0$ is periodic, then $\mu \leq \frac{D\pi^2}{(\min L)^2}$.
\end{example}

Let us continue to consider to the linear equation on the time-dependent interval $A(t)<x<A(t)+L(t)$ where $L(t)>0$ and $A(t)$ are $T$-periodic and belong to $C^{2+\alpha}([0,T])$ for some $\alpha>0$. 
We shall apply the same changes of variables that were introduced in \cite{JA1}, \cite{JA2}. Namely, first we transform onto a fixed reference domain $0<\xi< L_0$ by letting $\xi=\left(\frac{x-A(t)}{L(t)}\right)L_0$. The solution $u(\xi,t)=\psi(x,t)$ then satisfies equation \eqref{eq_u_xi} with $f(u)=f'(0)u$, and $u(\xi,t)=0$ at $\xi=0$ and $\xi=L_0$.
Next, we let $w(\xi,t)=u(\xi,t)H(\xi,t)e^{-f'(0)t}$ where
\begin{equation}\label{eq_w_uHe}
H(\xi, t)=\left(\frac{L(t)}{L(0)}\right)^{1/2}\exp{\left( \int\limits_0^t \frac{\dot{A}(\zeta)^2}{4D} d\zeta+\frac{\xi^2 \dot{L}(t)L(t)}{4DL_0^2}+\frac{\xi\dot{A}(t)L(t)}{2DL_0}\right)}.
\end{equation}
As in \cite{JA1}, \cite{JA2}, we find that $w$ satisfies the problem
\begin{equation}\label{eq_w}
\frac{\partial w}{\partial t} = D \frac{L_0^2}{L(t)^2} \frac{\partial^2 w}{\partial \xi^2}+\left( \frac{\xi^2\ddot{L}(t)L(t)}{4DL_0^2} + \frac{\xi\ddot{A}(t)L(t)}{2DL_0} \right)w  \qquad \textrm{in } 0< \xi< L_0
\end{equation}
\begin{equation}\label{eq_w_BC}
w(\xi,t)=0 \qquad \textrm\qquad\textrm{at } \xi=0 \textrm{ and } \xi=L_0 .
\end{equation}
From \cite[Theorem 2.1]{JA2} we have the following result.
\begin{theorem}\label{theorem_Q}
Let $w(\xi,t) \geq 0$ satisfy \eqref{eq_w}, \eqref{eq_w_BC}, and assume 
that $C_1 \sin\left(\frac{\pi\xi}{L_0}\right) \leq w(\xi,0) \leq C_2 \sin\left(\frac{\pi\xi}{L_0}\right)$ for some $0<C_1\leq C_2$. Define
\begin{align}\label{eq_Q_def}
\overline{Q}(t)=\max_{0\leq\eta\leq 1} \left(\frac{\eta^2\ddot{L}(t)L(t)}{2} + \eta \ddot{A}(t)L(t) \right),\qquad
\underline{Q}(t)=-\min_{0\leq\eta\leq 1} \left( \frac{\eta^2\ddot{L}(t)L(t)}{2} + \eta \ddot{A}(t)L(t) \right).
\end{align}
Then for every $t\geq 0$,
\begin{equation}\label{eq_w_bounds_Q}
C_1 \sin\left(\frac{\pi\xi}{L_0}\right)e^{\int_0^t \left(-\frac{D\pi^2}{L(\zeta)^2}-\frac{\underline{Q}(\zeta)}{2D}\right)d\zeta}\leq w(\xi,t)\leq C_2\sin\left(\frac{\pi\xi}{L_0}\right)e^{\int_0^t \left(-\frac{D\pi^2}{L(\zeta)^2}+\frac{\overline{Q}(\zeta)}{2D}\right)d\zeta}.
\end{equation}
\end{theorem}
\begin{proof}
By the definitions of $\overline{Q}(t)$ and $\underline{Q}(t)$, and the equation \eqref{eq_w} satisfied by $w(\xi,t)\geq0$, we have
\begin{equation}
-\frac{\underline{Q}(t)}{2D}w(\xi,t)\leq \frac{\partial w}{\partial t}-D\frac{L_0^2}{L(t)^2}\frac{\partial^2 w}{\partial \xi^2}  \leq \frac{\overline{Q}(t)}{2D}w(\xi,t) \qquad\textrm{in } 0<\xi<L_0.
\end{equation}
By applying the parabolic comparison principle on $[0,L_0]\times[0,t]$, we obtain the lower and upper bounds stated.
\end{proof}

Let $\mu_u=\mu$ denote the principal periodic eigenvalue of the operator
which acts on $u$ in equation \eqref{eq_u_xi}, and let $\mu_w$ denote the principal periodic eigenvalue of the operator which acts on $w$ in equation \eqref{eq_w}.
As in equation \eqref{eq_u_phi_mu} we know that
\begin{equation}\label{eq_u_w_mu}
u(\xi,t)=\overline{\underline{O}} (\phi_u(\xi,t)e^{(f'(0)-\mu_u)t}), \qquad
w(\xi,t)=\overline{\underline{O}} (\phi_w(\xi,t)e^{-\mu_w t})
\end{equation}
where $\phi_u(\xi,t)$ and $\phi_w(\xi,t)$ are the principal periodic eigenfunctions associated with $\mu_u$ and $\mu_w$, and where we use the notation $u_1=\overline{\underline{O}}(u_2)$ to mean that $u_1=O(u_2)$ and $u_2=O(u_1)$.
The relationship between $\mu_u$ and $\mu_w$ is given in the following lemma.
\begin{lemma}\label{lemma_mu_u_w}
Let $\mu_u=\mu$ denote the principal periodic eigenvalue of the operator which acts on $u$ in equation \eqref{eq_u_xi}, and let $\mu_w$ denote the principal periodic eigenvalue of the operator which acts on $w$ in equation \eqref{eq_w}. Then
\begin{equation}\label{eq_mu_u_w}
\mu_u = \mu_w+\frac{1}{T}\int_0^T \frac{\dot{A}(t)^2}{4D}dt.
\end{equation}
\end{lemma}
\begin{proof}
Consider the function $H(\xi,t)$ given by \eqref{eq_w_uHe}, which occurs in the change of variables from $u$ to $w$. Since $L(t)>0$ and $A(t)$ are both periodic, note that 
\begin{equation}
\left(\frac{L(t)}{L(0)}\right)^{1/2}\exp{\left(\frac{\xi^2 \dot{L}(t)L(t)}{4DL_0^2}+\frac{\xi\dot{A}(t)L(t)}{2DL_0}\right)} = \overline{\underline{O}}(1)
\end{equation}
in the sense that the left-hand side has a finite upper bound a positive lower bound, uniformly in $t\geq0$, $0\leq\xi\leq L_0$. Therefore from the change of variables and the periodicity of $\dot{A}(t)$, we have
\begin{equation}\label{eq_u_w_periodic}
u(\xi,t)= \overline{\underline{O}}\left( w(\xi,t)e^{f'(0)t}\exp\left( -\int\limits_0^t \frac{\dot{A}(\zeta)^2}{4D} d\zeta\right) \right) =\overline{\underline{O}}\left( w(\xi,t)\exp\left(f'(0)t -\frac{t}{T} \int\limits_0^T \frac{\dot{A}(\zeta)^2}{4D} d\zeta\right) \right).
\end{equation}
The claimed relationship \eqref{eq_mu_u_w} then follows by combining \eqref{eq_u_w_periodic} with \eqref{eq_u_w_mu}.
\end{proof}

We also derive upper and lower bounds on $\mu_w$ and hence, via equation \eqref{eq_mu_u_w}, upper and lower bounds on $\mu_u$.
\begin{theorem}\label{theorem_mu_bounds2}
Let $\mu=\mu_u$ denote the principal periodic eigenvalue of the operator
which acts on $u$ in equation \eqref{eq_u_xi}. Let $\overline{Q}(t)$, $\underline{Q}(t)$ be given by equation \eqref{eq_Q_def}. Then
\begin{equation}\label{eq_mu_bounds_Q_u}
\frac{1}{T}\int_0^T \left(\frac{D\pi^2}{L(t)^2} +\frac{\dot{A}(t)^2}{4D} -\frac{\overline{Q}(t)}{2D}\right) dt \leq\mu \leq \frac{1}{T}\int_0^T \left(\frac{D\pi^2}{L(t)^2}+\frac{\dot{A}(t)^2}{4D}+\frac{\underline{Q}(t)}{2D}\right) dt.
\end{equation}
\end{theorem}
\begin{proof}
By comparing \eqref{eq_u_w_mu} with \eqref{eq_w_bounds_Q} it follows that
\begin{equation}
\frac{1}{T}\int_0^T \left(\frac{D\pi^2}{L(t)^2} -\frac{\overline{Q}(t)}{2D}\right) dt \leq\mu_w \leq \frac{1}{T}\int_0^T \left(\frac{D\pi^2}{L(t)^2}+\frac{\underline{Q}(t)}{2D}\right) dt.
\end{equation}
The result then follows immediately by combining this with Lemma \ref{lemma_mu_u_w}.
\end{proof}
\begin{remark}
If we were instead interested in the principal periodic eigenvalue $\mu_V$ of an operator
$\frac{\partial}{\partial t}-D\nabla^2+V(x,t)$ on $\Omega(t)$, where $V$ was a continuous function on $\Omega(t)$ and periodic in $t$ with the same period as the domain, then a similar result could be derived provided that $\overline{Q}(t)$ and $\underline{Q}(t)$ were adjusted appropriately to also include the terms from $V$.
\end{remark}
\begin{example} \label{example_periodic_L_om_eps}
Let $L(t)=L_0(1+\varepsilon\sin(\omega t))$ with $\omega>0$, $0<\varepsilon<1$, and consider the domain $\Omega(t)=(0,L(t))$ which has period $T=\frac{2\pi}{\omega}$. Let $\mu=\mu_u$ be the principal periodic eigenvalue of $\Omega(t)$. We shall apply Theorems \ref{theorem_mu_bounds1a}, \ref{theorem_mu_bounds1b} and \ref{theorem_mu_bounds2} to give some bounds on $\mu$. First we must consider
\begin{equation}\label{eq_s(t)_L_om_eps}
s(t):=\int\limits_0^t \frac{L_0^2}{L(\zeta)^2}d\zeta=\int\limits_0^t \frac{1}{(1+\varepsilon\sin(\omega \zeta))^2}d\zeta.
\end{equation}
This integral \eqref{eq_s(t)_L_om_eps} can be calculated exactly. For $-\frac{\pi}{\omega}< t<\frac{\pi}{\omega}$,
\begin{align}
s(t)=&\frac{2}{\omega(1-\varepsilon^2)^{3/2}}\left(\arctan\left(\frac{\tan(\frac{\omega t}{2})+\varepsilon}{\sqrt{1-\varepsilon^2}}\right) - \arctan\left(\frac{\varepsilon}{\sqrt{1-\varepsilon^2}}\right)\right) -\frac{2\varepsilon}{\omega(1-\varepsilon^2)}  \nonumber \\
&+\frac{2\varepsilon^2\tan(\frac{\omega t}{2})+2\varepsilon}{\omega(1-\varepsilon^2)\left((\tan(\frac{\omega t}{2})+\varepsilon)^2+1-\varepsilon^2\right)}.
\end{align}
For $t=\pm \frac{\pi}{\omega}$,
\begin{align}
s\left(\pm \frac{\pi}{\omega}\right)=\frac{2}{\omega(1-\varepsilon^2)^{3/2}}\left(\pm\frac{\pi}{2} - \arctan\left(\frac{\varepsilon}{\sqrt{1-\varepsilon^2}}\right)\right) -\frac{2\varepsilon}{\omega(1-\varepsilon^2)},
\end{align}
and for $t> \frac{\pi}{\omega}$ we can use the fact that the integrand of $s(t)$ is periodic. Note that at $t=\frac{2\pi}{\omega}$,
\begin{equation}
s\left(\frac{2\pi}{\omega}\right)=\frac{2\pi}{\omega}\frac{1}{ (1-\varepsilon^2)^{3/2}}.
\end{equation}
By the periodicity, it follows that
\begin{equation}\label{eq_s(t)_periodic}
s(t)=\int\limits_0^t \frac{L_0^2}{L(\zeta)^2}d\zeta=\frac{t}{(1-\varepsilon^2)^{3/2}}+O(1) \qquad\textrm{as }t\rightarrow\infty .
\end{equation}
Therefore we conclude from Theorems \ref{theorem_mu_bounds1a} and \ref{theorem_mu_bounds1b} that
\begin{equation}\label{eq_lower}
\frac{D\pi^2}{L_0^2(1-\varepsilon^2)^{3/2}} \leq \mu\leq \frac{D\pi^2}{L_0^2(1-\varepsilon)^2}.
\end{equation}
This means that, regardless of the frequency $\omega$, the solution to the linear version of \eqref{eq_psi}, \eqref{eq_psi_BC}, will tend to zero if
$f'(0)<\frac{D\pi^2}{L_0^2(1-\varepsilon^2)^{3/2}}$, or tend to infinity if $f'(0)> \frac{D\pi^2}{L_0^2(1-\varepsilon)^2}$.
 
To apply Theorem \ref{theorem_mu_bounds2}, we calculate the $\frac{2\pi}{\omega}$-periodic functions $\overline{Q}(t)$ and $\underline{Q}(t)$ as defined in \eqref{eq_Q_def}:
\begin{equation}
\overline{Q}(t)=\begin{cases} 0 &\textrm{for }0\leq t\leq \frac{\pi}{\omega}\\
\ds{-\frac{L_0^2\varepsilon\omega^2}{2}\sin(\omega t)(1+\varepsilon\sin(\omega t))}&\textrm{for }\frac{\pi}{\omega}\leq t\leq \frac{2\pi}{\omega}\\
\end{cases}
\end{equation}
\begin{equation}
\underline{Q}(t)=\begin{cases}\ds{\frac{L_0^2\varepsilon\omega^2}{2}\sin(\omega t)(1+\varepsilon\sin(\omega t))}&\textrm{for } 0\leq t\leq \frac{\pi}{\omega}\\
0&\textrm{for }\frac{\pi}{\omega}\leq t\leq \frac{2\pi}{\omega}\\
\end{cases}
\end{equation}
and therefore
\begin{equation}
\int_0^{\frac{2\pi}{\omega}} \frac{\overline{Q}(\zeta)}{2D}d\zeta = \frac{L_0^2\varepsilon\omega}{2D}\left(1-\frac{\varepsilon\pi}{4}\right)
\qquad\textrm{and}\qquad
\int_0^{\frac{2\pi}{\omega}} \frac{\underline{Q}(\zeta)}{2D}d\zeta = \frac{L_0^2\varepsilon\omega}{2D}\left(1+\frac{\varepsilon\pi}{4}\right).
\end{equation}
By equation \eqref{eq_mu_bounds_Q_u} together with \eqref{eq_s(t)_periodic} we deduce that
\begin{equation}\label{eq_ex1_Q}
\frac{D\pi^2}{L_0^2(1-\varepsilon^2)^{3/2}} -\frac{L_0^2\varepsilon\omega^2}{4\pi D}\left(1-\frac{\varepsilon\pi}{4}\right) \leq \mu\leq 
\frac{D\pi^2}{L_0^2(1-\varepsilon^2)^{3/2}}+ \frac{L_0^2\varepsilon\omega^2}{4\pi D}\left(1+\frac{\varepsilon\pi}{4}\right),
\end{equation}
although we note that in this case the lower bound in \eqref{eq_lower} is better.
\end{example}
\begin{remark}
Let $\mu(\omega)$ be the principal periodic eigenvalue of $\Omega(t)=(0,L(t))$ with $L(t)$ as in Example \ref{example_periodic_L_om_eps} for some fixed $\varepsilon \in(0,1)$. The bounds \eqref{eq_ex1_Q} imply that
\begin{equation}\label{eq_remark_om1}
\mu(\omega)= \frac{D\pi^2}{L_0^2(1-\varepsilon^2)^{3/2}}+O(\omega^2)=\frac{1}{T}\int\limits_0^T\frac{D\pi^2}{L(t)^2}dt+O(\omega^2) \quad \textrm{ as }\omega \rightarrow 0.
\end{equation}
In fact, for any interval of the form $A(t)<x<A(t)+L(t)$ where $A(t)=A_1 \left(\frac{\omega t}{2\pi}\right)$ and $L(t)=L_1\left(\frac{\omega t}{2\pi}\right)$ for some smooth and $1$-periodic functions $L_1>0$ and $A_1$, we can use \eqref{eq_mu_bounds_Q_u} to deduce the limit as $\omega \rightarrow 0$. Indeed since $\dot{A}^2$, $\underline{Q}$ and $\overline{Q}$ (as defined in \eqref{eq_Q_def}) are all $O(\omega^2)$ as $\omega\rightarrow 0$, we can conclude from \eqref{eq_mu_bounds_Q_u} that
\begin{equation}\label{eq_remark_om2}
\mu(\omega) =\int_0^1 \frac{D\pi^2}{L_1(s)^2}ds + O(\omega^2) \qquad \textrm{as }\omega\rightarrow 0 .
\end{equation}
\end{remark}

In the next section we shall see that the limit $\lim_{\omega\rightarrow0} \mu(\omega)$ from \eqref{eq_remark_om1} and \eqref{eq_remark_om2} is an instance of a more general property, which is valid for $\frac{2\pi}{\omega}$-periodic domains $\Omega(t)$ in any dimension, as $\omega \rightarrow 0$.
\section{Dependence of $\mu$ on the frequency $\omega$}\label{section_mu_omega}
\subsection{Converting to a $1$-periodic problem}
In this section we consider the principal periodic eigenvalue $\mu$ as a function of the frequency $\omega=\frac{2\pi}{T}$. We consider a $1$-periodic domain $\tilde{\Omega}(s)$ and let $\mu=\mu(\omega)$ be the principal periodic eigenvalue associated with the domain
\begin{equation}
\Omega(t)=\tilde{\Omega}\left(\frac{\omega t}{2\pi}\right).
\end{equation}
Note that the map $h(\cdot ,t):\Omega(t)\rightarrow \Omega_0$ which we used in the change of variables can now be expressed as $h(\cdot ,t)=\tilde{h}(\cdot ,\frac{\omega t}{2\pi})$, for a $1$-periodic map $\tilde{h}(\cdot ,s): \tilde{\Omega}(s)\rightarrow \Omega_0$.
If we change variables from $t$ to $s=\frac{\omega t}{2\pi}$ in \eqref{eq_u_ij}, \eqref{eq_u_ij_BC}, then the operator $\frac{\partial}{\partial t}- \mathcal{L}(\xi,t)$ becomes an operator of the form $\frac{\omega}{2\pi}\frac{\partial }{\partial s}- \mathcal{L}_{\omega}(\xi,s)$
where
\begin{equation}\label{eq_Lomega}
\mathcal{L}_{\omega}(\xi,s) = \sum_{i,j}\tilde{a}_{ij}(\xi,s)\frac{\partial^2 u}{\partial \xi_i\partial \xi_j} +\sum_j \left(\frac{\omega}{2\pi}\tilde{b}_j(\xi,s) +\tilde{c}_j(\xi,s)\right)\frac{\partial}{\partial \xi_j},
\end{equation}
\begin{equation}\label{eq_aij_bj_cj_s}
\tilde{a}_{ij}(\xi,s)= \sum_{k}D\left(\frac{\partial \tilde{h}_i}{\partial x_k}\frac{\partial \tilde{h}_j}{\partial x_k}\right), \qquad
\tilde{b}_j(\xi,s)= -\frac{\partial \tilde{h}_j}{\partial s}, \qquad
\tilde{c}_j(\xi,s)=D\nabla^2 \tilde{h}_j.
\end{equation}
So the principal periodic eigenvalue of $\Omega(t)=\tilde{\Omega}\left(\frac{\omega t}{2\pi}\right)$ is the same as the principal periodic eigenvalue $\mu(\omega)$ of the problem
\begin{equation}\label{eq_omega_s1}
\frac{\omega}{2\pi}\frac{\partial \phi}{\partial s}- \mathcal{L}_{\omega}(\xi,s)\phi = \mu(\omega)\phi(\xi,s) \qquad \xi\in\Omega_0, s\in[0,1]
\end{equation}
\begin{equation}\label{eq_omega_s2}
\phi(\xi,s)=0 \qquad \xi\in\partial\Omega_0, s\in[0,1]
\end{equation}
\begin{equation}\label{eq_omega_s3}
\phi(\xi,s)\equiv \phi(\xi,s+1) \qquad \xi\in\Omega_0.
\end{equation}
For any $\omega>0$, the coefficients of $\mathcal{L}_{\omega}(\xi,s)$ are $1$-periodic in $s$. However, the term $\frac{\omega}{2\pi} \tilde{b}_j(\xi,s) \frac{\partial}{\partial \xi_j}$ in \eqref{eq_Lomega} still depends on, and scales with, $\omega$. The source of this term is the coefficient involving $\frac{\partial h_j}{\partial t}=\frac{\omega}{2\pi}\frac{\partial \tilde{h}_j}{\partial s}$ which we get when we transform the time-dependent domain $\tilde{\Omega}\left(\frac{\omega t}{2\pi}\right)$ into the fixed reference domain $\Omega_0$.

In the paper \cite{LiuLouPengZho}, Liu, Lou, Peng and Zhou consider parabolic equations with periodic coefficients, and they investigate how the principal periodic eigenvalue varies with respect to the frequency. However, the coefficients in their equation are independent of the frequency $\omega$ except where it appears inside the periodic functions as $\omega t$. Therefore, after the change of time variables to give an operator with $1$-periodic coefficients, the problems they consider have the form
\begin{equation}\label{eq_Peng1}
\frac{\omega}{2\pi}\frac{\partial \hat{\phi}}{\partial s}- \mathcal{\hat{L}}(\xi,s)\hat{\phi} = \hat{\lambda}(\omega)\hat{\phi}(\xi,s) \qquad \xi\in\Omega_0, s\in[0,1]
\end{equation}
\begin{equation}\label{eq_Peng2}
\hat{\phi}(\xi,s)=0 \qquad \xi\in\partial\Omega_0, s\in[0,1]
\end{equation}
\begin{equation}\label{eq_Peng3}
\hat{\phi}(\xi,s)\equiv \hat{\phi}(\xi,s+1) \qquad \xi\in\Omega_0,
\end{equation}
where the coefficients of the operator $\mathcal{\hat{L}}(\xi,s)$ are $1$-periodic in $s$ and do not depend on $\omega$. We shall use methods from \cite{LiuLouPengZho} in the proofs of Theorem \ref{theorem_om_tends_0} and Theorem \ref{theorem_monotonicity_omega_new}, but the methods have to be adapted, since the operator $\mathcal{L}_{\omega}(\xi,s)$ depends on $\omega$ through the $\frac{\omega}{2\pi} \tilde{b}_j(\xi,s) \frac{\partial}{\partial \xi_j}$ term.

\subsection{Asymptotic behaviour of $\mu(\omega)$ as $\omega\rightarrow0$}
In this section, we consider the limit of $\mu(\omega)$ as $\omega\rightarrow0$. The proof of Theorem \ref{theorem_om_tends_0} is essentially the same as that used in \cite[Theorem 1.3(i)]{LiuLouPengZho} to prove equation \eqref{eq_lambda_om0}, however a slight generalisation is needed to allow for the $\omega$ dependence of the coefficients in $-\mathcal{L}_{\omega}(\xi,s)$. For completeness, we give the proof here.
\begin{theorem}\label{theorem_om_tends_0}
Let $\Omega_0$ be a smooth bounded domain, and for each $s \in[0,1]$ and $\omega\geq0$ let $-\mathcal{L}_{\omega}(\xi,s)$ be as defined in equation \eqref{eq_Lomega}. Assume the coefficients $\tilde{a}_{ij}$, $\tilde{b}_j$, $\tilde{c}_j$ belong to $C^{1+\alpha, 1+\alpha}(\overline{\Omega_0}\times[0,1])$ for some $\alpha>0$.

For $s \in[0,1]$ and $\omega\geq0$, let $\lambda^{0}(s,\omega)$ be the principal eigenvalue of the elliptic operator $-\mathcal{L}_{\omega}(\xi,s)$ on $\Omega_0$, with zero Dirichlet conditions on $\partial\Omega_0$. For $\omega>0$, let $\mu(\omega)$ be the principal periodic eigenvalue of \eqref{eq_omega_s1}, \eqref{eq_omega_s2}, \eqref{eq_omega_s3}. Then
\begin{equation}\label{eq_lambda0_omega}
\lim_{\omega\rightarrow 0}\mu(\omega)= \int_0^1 \lambda^{0}(s,0)ds.
\end{equation}
\end{theorem}
\begin{proof}
For each $s \in[0,1]$ and $\omega\geq0$, let $\lambda^{0}(s,\omega)$ and $\phi^{0}(\xi; s,\omega)$ be the principal eigenvalue and eigenfunction of the elliptic operator $-\mathcal{L}_{\omega}(\xi,s)$ on $\Omega_0$, with zero Dirichlet conditions on $\partial\Omega_0$, and normalised to $\vert\vert \phi^{0}(\cdot \ ; s,\omega) \vert\vert_{L^2(\Omega_0)} =1$. As in \cite{LiuLouPengZho}, note that for any $\xi\in\overline{\Omega_0}$ and $\omega\geq 0$, both $\phi^{0}(\xi;s,\omega)$ and $\nabla\phi^{0}(\xi;s,\omega)$ are $C^1$ and $1$-periodic in $s$. We also note here that they, and $\lambda^{0}(s,\omega)$, depend continuously on $\omega$.

For $\omega>0$, define
\begin{equation}\label{eq_overline_phi}
\overline{\phi}_{\omega}(\xi,s) = \phi^{0}(\xi; s,\omega) \rho_{\omega}(s)
\end{equation}
where
\begin{equation}
\rho_{\omega}(s)=\exp\left( \frac{2\pi}{\omega}\left( s \int_0^1 \lambda^{0}(\tau,\omega) d\tau -  \int_0^s \lambda^{0}(\tau,\omega) d\tau \right)\right).
\end{equation}
Note that $\rho_{\omega}>0$, $\rho_{\omega}$ is periodic with period $1$, and satisfies
\begin{equation}\label{eq_rho_om}
\frac{\omega}{2\pi} \frac{d \rho_{\omega}}{ds}=\left( \int_0^1 \lambda^{0}(\tau,\omega) d\tau -  \lambda^{0}(s,\omega) \right) \rho_{\omega}(s).
\end{equation}
We shall show that given $\varepsilon>0$, $\omega_{\varepsilon}>0$ can be chosen small enough such that
\begin{equation}\label{eq_ineq_rho_om}
\left( \int_0^1 \lambda^{0}(\tau,\omega) d\tau-\varepsilon\right) \overline{\phi}_{\omega} \leq 
\frac{\omega}{2\pi}\frac{\partial \overline{\phi}_{\omega}}{\partial s}- \mathcal{L}_{\omega}(\xi,s)\overline{\phi}_{\omega} \leq
\left( \int_0^1 \lambda^{0}(\tau,\omega) d\tau+\varepsilon\right) \overline{\phi}_{\omega}
\end{equation}
for all $0<\omega\leq \omega_{\varepsilon}$.
Then since $\overline{\phi}_{\omega}$ is positive, $1$-periodic in $s$, and satisfies the Dirichlet boundary conditions on $\partial\Omega_0$, it follows from
\cite[Proposition 2.1]{PengZhao} that
\begin{equation}
\int_0^1 \lambda^{0}(\tau,\omega) d\tau-\varepsilon \leq \mu(\omega) \leq \int_0^1 \lambda^{0}(\tau,\omega) d\tau+\varepsilon  \qquad\textrm{for all }0<\omega\leq \omega_{\varepsilon},
\end{equation}
and so we reach the conclusion
\begin{equation}\label{eq_om_tends_0_lemma}
\lim_{\omega\rightarrow 0} \left( \mu(\omega) - \int_0^1 \lambda^{0}(s,\omega) ds \right) =0.
\end{equation}
Finally, since $\lambda^{0}(s,\omega)$ depends continuously on $\omega$, \eqref{eq_om_tends_0_lemma} implies \eqref{eq_lambda0_omega}.

It remains to show that $\omega_{\varepsilon}>0$ can be chosen such that \eqref{eq_ineq_rho_om} holds for all $0<\omega\leq \omega_{\varepsilon}$. Using \eqref{eq_overline_phi}, \eqref{eq_rho_om}, and the fact that $\phi^{0}(\xi; s,\omega)$ is an eigenfunction of $-\mathcal{L}_{\omega}(\xi,s)$ with eigenvalue $\lambda^{0}(s,\omega)$, we calculate
\begin{align}
\frac{\omega}{2\pi}\frac{\partial \overline{\phi}_{\omega}}{\partial s}- \mathcal{L}_{\omega}(\xi,s)\overline{\phi}_{\omega} &=
\frac{\omega}{2\pi}\frac{\partial \phi^{0}(\xi; s,\omega)}{\partial s} \rho_{\omega} + \frac{\omega}{2\pi} \frac{d \rho_{\omega}}{ds}\phi^{0}(\xi; s,\omega) + \lambda^{0}(s,\omega)\phi^{0}(\xi; s,\omega)\rho_{\omega} \\
&= \left(\frac{\omega}{2\pi}\frac{\partial \phi^{0}(\xi; s,\omega)}{\partial s} + \int_0^1 \lambda^{0}(\tau,\omega) d\tau \ \phi^{0}(\xi; s,\omega) \right) \rho_{\omega}.
\end{align}
Therefore \eqref{eq_ineq_rho_om} will hold provided that we can choose $\omega_{\varepsilon}>0$ such that
\begin{equation}\label{eq_ineq_rho_om2}
\frac{\omega}{2\pi}\left\vert \frac{\partial \phi^{0}(\xi; s,\omega)}{\partial s} \right\vert \leq \varepsilon \phi^{0}(\xi; s,\omega) \qquad\textrm{for all } \xi\in\Omega_0,\ s\in[0,1],\ 0<\omega\leq \omega_{\varepsilon}.
\end{equation}
Since $\phi^{0}(\xi; s,\omega)$ is positive in $\Omega_0$, we know that $\frac{\frac{\partial \phi^{0}}{\partial s}(\xi; s,\omega)}{\phi^{0}(\xi; s,\omega)}$ is finite for each $\xi$ in $\Omega_0$. For any point $\xi_0\in\partial \Omega_0$, with outward normal $\nu$, we may consider a sequence $\xi\in\Omega_0$, $\xi \rightarrow \xi_0$ with $\frac{\xi-\xi_0}{\vert \xi-\xi_0\vert}\cdot\nu \nrightarrow0$. Hopf's Lemma implies that $\nabla\phi^{0}(\xi_0; s,\omega)\cdot \nu \neq 0$, and then by l'H{\^{o}}pital's rule, we have
\begin{equation}\label{eq_lHopital}
\lim_{\xi \rightarrow \xi_0} \frac{\frac{\partial \phi^{0}}{\partial s}(\xi; s,\omega)}{\phi^{0}(\xi; s,\omega)} =\frac{\nabla\frac{\partial \phi^{0}}{\partial s}(\xi_0; s,\omega)\cdot \nu}{\nabla\phi^{0}(\xi_0; s,\omega)\cdot \nu}= O(1).
\end{equation}
By continuity, and by the compactness of $\overline{\Omega_0}\times[0,1]\times[0,1]$, we find that 
$\frac{\frac{\partial \phi^{0}}{\partial s}(\xi; s,\omega)}{\phi^{0}(\xi; s,\omega)}$ is bounded uniformly with respect to $(\xi_0,s,\omega)\in \Omega_0\times[0,1]\times[0,1]$. Therefore $\omega_{\varepsilon}>0$ can be chosen to satisfy \eqref{eq_ineq_rho_om2}.
\end{proof}

\begin{corollary}\label{corollary_omega_tends_0}
Let $\tilde{\Omega}(s)$ be a smooth bounded domain which varies smoothly and $1$-periodically with $s$, and let $\mu(\omega)$ be the principal periodic eigenvalue of $\Omega(t)=\tilde{\Omega}\left(\frac{\omega t}{2\pi}\right)$. For each $0\leq s\leq 1$, let $\lambda (\tilde{\Omega}(s))$ be the principal Dirichlet eigenvalue of $-\nabla^2$ on $\tilde{\Omega}(s)$. Then
\begin{equation}\label{eq_omega_tends_0}
\lim_{\omega\rightarrow 0}\mu(\omega) =\int_0^1 D\lambda(\tilde{\Omega}(s))ds.
\end{equation}
\end{corollary}
\begin{proof}
For each $0\leq s\leq 1$, the change of variables $\tilde{h}$ from $\tilde{\Omega}(s)$ to $\Omega_0$ transforms the operator $D\nabla^2$ on $\tilde{\Omega}(s)$ to $\sum_{i,j,k} D\left(\frac{\partial \tilde{h}_i}{\partial x_k}\frac{\partial \tilde{h}_j}{\partial x_k}\right)\frac{\partial^2 }{\partial \xi_i\partial \xi_j} +\sum_jD\nabla^2 \tilde{h}_j\frac{\partial}{\partial \xi_j}$ on $\Omega_0$. By equations \eqref{eq_Lomega} and \eqref{eq_aij_bj_cj_s}, this is precisely $\mathcal{L}_{0}(\xi,s)$ (i.e. $\mathcal{L}_{\omega}(\xi,s)$ with $\omega=0$).
So we have
\begin{equation}
D\lambda (\tilde{\Omega}(s))=\lambda^{0}(s,0)
\end{equation}
and then \eqref{eq_omega_tends_0} is equivalent to \eqref{eq_lambda0_omega}.
\end{proof}
\begin{remark}
Recall that $\int_0^1 D\lambda(\tilde{\Omega}(s))ds$ is also a lower bound for $\mu(\omega)$ for every $\omega>0$ (see Theorem \ref{theorem_mu_bounds1a}).
\end{remark}
\begin{remark}
A result corresponding to Corollary \ref{corollary_omega_tends_0} would also hold for the principal periodic eigenvalue $\mu_V$ of an operator
$\frac{\partial}{\partial t}-D\nabla^2+V(x,t)$ on $\Omega(t)$, where $V$ was a continuous function on $\Omega(t)$ and periodic in $t$ with the same period as the domain. Writing $V(x,t)=\tilde{V}\left(x,\frac{\omega t}{2\pi}\right)$, the same proof would show then that
\begin{equation}\nonumber
\lim_{\omega\rightarrow 0}\mu_V(\omega) =\int_0^1 \lambda_{D,\tilde{V}}(\tilde{\Omega}(s))ds
\end{equation}
where for each $s$, $\lambda_{D,\tilde{V}}(\tilde{\Omega}(s))$ was the principal Dirichlet eigenvalue of the elliptic operator $-D\nabla^2+\tilde{V}(x,s)$ on $\tilde{\Omega}(s)$.
\end{remark}
\subsection{Asymptotic behaviour of $\mu(\omega)$ as $\omega\rightarrow \infty$}
For one-dimensional time-periodic domains $\Omega(t)=(A(t),A(t)+L(t))$, we shall give conditions under which $\mu(\omega)$ does and does not remain bounded as $\omega \rightarrow\infty$.
\begin{theorem} \label{theorem_om_tends_infty}
Let $l(\cdot)$ and $a(\cdot)$ be $1$-periodic functions, belonging to $C^{2+\alpha}([0,1])$ for some $\alpha>0$, and with $\min_{[0,1]} l=1$, $\max_{[0,1]} \vert a\vert =1$. For some $L_0>0$, $A_0\geq0$, $\omega>0$, let
\begin{equation}\label{eq_LA_la}
L(t)=L_0 l\left(\frac{\omega t}{2\pi}\right), \qquad
A(t)=A_0 a\left(\frac{\omega t}{2\pi}\right),
\end{equation}
and let $\mu(\omega)=\mu_u(\omega)$ be the principal periodic eigenvalue associated with $\Omega(t)=(A(t),A(t)+L(t))$. Then $\mu(\omega)=O(\omega^2)$ as $\omega \rightarrow\infty$, and if $a(\cdot)$ is constant, then $\mu(\omega)=O(1)$ as $\omega \rightarrow\infty$. Moreover, if $a(\cdot)$ is non-constant, there exist constants $C_1$, $C_2$ depending only on the functions $l$ and $a$ such that
\begin{enumerate}
\item If $\frac{A_0}{L_0}<C_1$ then $\mu(\omega)=O(1)$ as $\omega \rightarrow\infty$.
\item If $\frac{A_0}{L_0}>C_2$ then $\mu(\omega)=\underline{\overline{O}}(\omega^2)$ as $\omega \rightarrow\infty$.
\end{enumerate}
\end{theorem}
As before, the notation $\mu(\omega)=\underline{\overline{O}}(\omega^2)$ is used to mean that $\mu(\omega)$ is `exactly of order' $\omega^2$ in the sense that $\mu(\omega)=O(\omega^2)$ and $\omega^2=O(\mu(\omega))$ as $\omega\rightarrow\infty$.
\begin{proof}
If $\max_{s\in[0,1]} (A_0 a(s))< \min_{s\in[0,1]}(A_0 a(s)+L_0 l(s))$ then by Theorems \ref{theorem_mu_bounds1a} and \ref{theorem_mu_bounds1b}, we have lower and upper bounds on $\mu(\omega)$ which are independent of $\omega>0$:
\begin{equation}
\frac{D\pi^2}{L_0^2}\int_0^1 \frac{1}{l(s)^2}ds \leq \mu(\omega) \leq \frac{D\pi^2}{(\min_{[0,1]}(A_0 a+L_0 l) - \max_{[0,1]} (A_0 a))^2}.
\end{equation}
So, $\mu(\omega)=O(1)$ as $\omega \rightarrow\infty$ as long as $\max_{[0,1]} (A_0 a)< \min_{[0,1]}(A_0 a+L_0 l)$. If $a(\cdot)$ is constant then this will be satisfied because, by assumption, $\min_{[0,1]} (L_0 l)>0$. If $a(\cdot)$ is non-constant, then a sufficient condition is that
\begin{equation}
\frac{A_0}{L_0}<\frac {\min l}{\max a -\min a}.
\end{equation}
Next, in order to prove the other claimed properties, we shall consider the bounds \eqref{eq_mu_bounds_Q_u} that were proved in Theorem \ref{theorem_mu_bounds2}. Define non-negative constants $c_1$, $c_2$, $c_3$, $c_4$, $c_5$, $c_6$ in terms of the functions $l$ and $a$ as follows:
\begin{align}
c_1=\int_0^1 \frac{1}{l(s)^2}ds, \qquad c_2=\int_0^1 a'(s)^2 ds, \qquad c_3=\int_0^1 l(s)[a''(s)]^{+} ds, \qquad \nonumber\\ c_4=\int_0^1 l(s)[l''(s)]^{+} ds, \qquad
c_5=\int_0^1 l(s)[a''(s)]^{-} ds, \qquad c_6=\int_0^1 l(s)[l''(s)]^{-} ds.
\end{align}
Then note that
\begin{equation}
\frac{1}{T}\int_0^T \frac{D\pi^2}{L(t)^2}dt=\frac{D\pi^2}{L_0^2} c_1,
\end{equation}
\begin{equation}
\frac{1}{T}\int_0^T \frac{\dot{A}(t)^2}{4D}dt=\left(\frac{\omega}{2\pi}\right)^2 \frac{A_0^2}{4D} c_2,
\end{equation}
\begin{equation}
0\leq \frac{1}{T}\int_0^T\frac{\overline{Q}(t)}{2D} dt \leq \left(\frac{\omega}{2\pi}\right)^2 \left(\frac{A_0L_0}{2D}c_3 +\frac{L_0^2}{4D}c_4 \right) ,
\end{equation}
\begin{equation}
0\leq \frac{1}{T}\int_0^T\frac{\underline{Q}(t)}{2D}dt \leq
\left(\frac{\omega}{2\pi}\right)^2 \left( \frac{A_0L_0}{2D}c_5 + \frac{L_0^2}{4D}c_6\right).
\end{equation}
Therefore Theorem \ref{theorem_mu_bounds2} implies that
\begin{equation}\label{eq_bounds_c1to6}
\frac{D\pi^2}{L_0^2}c_1 +\left(\frac{\omega}{2\pi}\right)^2\left(\frac{A_0^2}{4D}c_2 -\frac{A_0L_0}{2D}c_3 -\frac{L_0^2}{4D}c_4\right) \leq \mu(\omega) \leq
\frac{D\pi^2}{L_0^2}c_1 +\left(\frac{\omega}{2\pi}\right)^2\left(\frac{A_0^2}{4D}c_2 +\frac{A_0L_0}{2D}c_5 +\frac{L_0^2}{4D}c_6\right),
\end{equation}
which proves that $\mu(\omega)=O(\omega^2)$ as $\omega \rightarrow\infty$, for any $A_0$, $L_0$. Moreover, if $\frac{A_0^2}{4}c_2 -\frac{A_0L_0}{2}c_3 -\frac{L_0^2}{4}c_4 >0$ then $\mu(\omega)=\underline{\overline{O}}(\omega^2)$ as $\omega \rightarrow\infty$. If $a(\cdot)$ is non-constant then $c_2\neq 0$ and so this inequality will hold for $\frac{A_0}{L_0}$ large enough (depending on $c_2$, $c_3$, $c_4$).
\end{proof}
In the following example we give these estimates explicitly.
\begin{example} \label{example_periodic_A_om}
Let $L_0>0$ be constant and $A(t)=A_0\sin(\omega t)$ for some $\omega>0$, $A_0>0$. Consider the $\frac{2\pi}{\omega}$-periodic domain $\Omega(t)=(A(t),A(t)+L_0)$ and let $\mu(\omega)=\mu_u(\omega)$ be the principal periodic eigenvalue on $\Omega(t)$. By Theorems \ref{theorem_mu_bounds1a} and \ref{theorem_mu_bounds1b} we conclude that
\begin{equation}\label{eq_Aom_1}
\frac{D\pi^2}{L_0^2} \leq \mu(\omega) \qquad\textrm{for every } \omega>0,
\end{equation}
\begin{equation}\label{eq_Aom_2}
\textrm{and if }2A_0<L_0, \qquad \mu(\omega)\leq \frac{D\pi^2}{(L_0-2A_0)^2}  \qquad\textrm{for every } \omega>0.
\end{equation}
To apply the bounds from Theorem \ref{theorem_mu_bounds2}, calculate
\begin{equation}\label{eq_Adot^2_periodic}
\int\limits_0^t \frac{\dot{A}(\zeta)^2}{4D} d\zeta=\frac{A_0^2\omega^2}{4D} \left(\frac{t}{2}+\frac{\sin(2\omega t)}{4\omega}\right).
\end{equation}
Also calculate the $\frac{2\pi}{\omega}$-periodic functions $\overline{Q}(t)$ and $\underline{Q}(t)$ as defined in \eqref{eq_Q_def}:
\begin{equation}
\overline{Q}(t)=\begin{cases} 0 &\textrm{for }0\leq t\leq \frac{\pi}{\omega}\\
\ds{-A_0L_0\omega^2\sin(\omega t)}&\textrm{for }\frac{\pi}{\omega}\leq t\leq \frac{2\pi}{\omega},
\end{cases}
\end{equation}
\begin{equation}
\underline{Q}(t)=\begin{cases}\ds{A_0L_0\omega^2\sin(\omega t)}&\textrm{for } 0\leq t\leq \frac{\pi}{\omega}\\
0&\textrm{for }\frac{\pi}{\omega}\leq t\leq \frac{2\pi}{\omega} ,\\
\end{cases}
\end{equation}
and so
\begin{equation}
\int_0^{\frac{2\pi}{\omega}} \frac{\overline{Q}(\zeta)}{2D}d\zeta =
\int_0^{\frac{2\pi}{\omega}} \frac{\underline{Q}(\zeta)}{2D}d\zeta = \frac{A_0 L_0\omega}{D}.
\end{equation}
By Theorem \ref{theorem_mu_bounds2} we deduce that
\begin{equation}\label{eq_ex2_Q}
\frac{D\pi^2}{L_0^2}+\frac{A_0^2\omega^2}{8D}-\frac{A_0 L_0 \omega^2}{2\pi D} \leq\mu(\omega) \leq \frac{D\pi^2}{L_0^2} +\frac{A_0^2\omega^2}{8D}+\frac{A_0 L_0 \omega^2}{2\pi D}.
\end{equation}
In agreement with Corollary \ref{corollary_omega_tends_0} and Theorem \ref{theorem_om_tends_infty}, the bounds \eqref{eq_Aom_1}, \eqref{eq_Aom_2} and \eqref{eq_ex2_Q} show that:
\begin{equation}
\mu=\frac{D\pi^2}{L_0^2}+O(\omega^2) \textrm{ as }\omega \rightarrow 0.
\end{equation}
\begin{equation}
\mu(\omega)=O(\omega^2)\textrm{ as }\omega\rightarrow \infty.
\end{equation}
\begin{equation}
\textrm{If } \quad \frac{A_0}{L_0} < \frac{1}{2} \quad\textrm{ then}
\quad \mu(\omega)=O(1)\textrm{ as }\omega\rightarrow \infty.
\end{equation}
\begin{equation}
\textrm{If } \quad \frac{A_0}{L_0}> \frac{4}{\pi}\quad\textrm{ then}
\quad \mu(\omega)=\underline{\overline{O}}(\omega^2)\textrm{ as }\omega\rightarrow \infty.
\end{equation}
It would be interesting to investigate the $\omega\rightarrow\infty$ limit in the intermediate parameter range $\frac{L_0}{2}\leq A_0 \leq \frac{4L_0}{\pi}$.
\end{example}

\subsection{Monotonicity of $\mu(\omega)$ with respect to $\omega>0$}
In this section, we prove first that the principal periodic eigenvalue associated with a $T$-periodic domain $\Omega(t)\subset \mathbb{R}^N$ is the same as the eigenvalue associated with the domain $\Omega(-t)$. We use this together with Lemma 2.1 and Theorem 1.1 from \cite{LiuLouPengZho} to show that, for $\omega>0$, the principal periodic eigenvalue $\mu(\omega)$ associated with the domain $\tilde{\Omega}\left(\frac{\omega t}{2\pi}\right)$ is monotonic non-decreasing with respect to $\omega$.

\begin{lemma} \label{lemma_omega_even}
Let $\Omega(t)$ be a $T$-periodic domain. Let $\mu_+$ be the principal periodic eigenvalue associated with $\Omega(t)$, and let $\mu_-$ be the principal periodic eigenvalue associated with $\Omega_-(t):=\Omega(-t)$. Then $\mu_+= \mu_-$.
\end{lemma}
\begin{proof}
The eigenvalues $\mu_+$, $\mu_-$ are principal periodic eigenvalues of problems of the form \eqref{eq_princ_eig1}, \eqref{eq_princ_eig2}, \eqref{eq_princ_eig3}, \eqref{eq_princ_eig4} on $\Omega_0$, with operators $\mathcal{L}_+$ and $\mathcal{L}_-$ coming from the changes of variables. In terms of the original co-ordinates, this means that there exist positive functions $\psi_+(x,t)$ on $\Omega(t)$ and $\psi_-(x,t)$ on $\Omega_-(t)$, which are $T$-periodic in $t$, and which satisfy
\begin{equation}\label{eq_psi+}
\frac{\partial \psi_+}{\partial t} = D \nabla^2 \psi_+ + \mu_+\psi_+  \qquad \textrm{for }x\in \Omega(t)
\end{equation}
\begin{equation}\label{eq_psi+_BC}
\psi_+(x,t)=0 \qquad\textrm{for }x\in \partial\Omega(t) 
\end{equation}
\begin{equation}\label{eq_psi-}
\frac{\partial \psi_-}{\partial t} = D \nabla^2 \psi_- + \mu_-\psi_-  \qquad \textrm{for }x\in \Omega_-(t)
\end{equation}
\begin{equation}\label{eq_psi-_BC}
\psi_-(x,t)=0 \qquad\textrm{for }x\in \partial\Omega_-(t) .
\end{equation}
Let $\overline{\psi}(x,t)=\psi_-(x,-t)$ for $x\in\Omega_-(-t)=\Omega(t)$. Then $\overline{\psi}(x,t)$ is positive on $\Omega(t)$ and $T$-periodic in $t$, and satisfies
\begin{equation}\label{eq_psi*}
-\frac{\partial \overline{\psi}}{\partial t} = D \nabla^2 \overline{\psi} + \mu_-\overline{\psi}  \qquad \textrm{for }x\in \Omega(t)
\end{equation}
\begin{equation}\label{eq_psi*_BC}
\overline{\psi}(x,t)=0 \qquad\textrm{for }x\in \partial\Omega(t) .
\end{equation}
If $I(t)= \int_{\Omega(t)} \psi_+(x,t)\overline{\psi}(x,t) dx$ then, using the zero Dirichlet boundary conditions on $\partial\Omega(t)$, we have
\begin{equation}
\frac{dI}{dt} =\int_{\Omega(t)} \left( \frac{\partial\psi_+}{\partial t}\overline{\psi} +   \frac{\partial\overline{\psi}}{\partial t}\psi_+ \right) dx.
\end{equation}
Using equations \eqref{eq_psi+} and \eqref{eq_psi*}, and integrating by parts, this becomes
\begin{equation}
\frac{dI}{dt} = ( \mu_+ - \mu_-)I(t).
\end{equation}
Therefore $I(t)= I(0)e^{( \mu_+ - \mu_-)t}$, but since $I(t)$ is periodic it must be that $ \mu_+ - \mu_- =0$.
\end{proof}
\begin{remark}
The proof of Lemma \ref{lemma_omega_even} not only shows that $\mu_+ = \mu_-$ but also that $I(t) \equiv I(0)$. That is, the integral $\int_{\Omega(t)} \psi_+(x,t)\psi_-(x,-t) dx$ is independent of $t$.
\end{remark}
By rescaling time to $s = \frac{\omega t}{2\pi}$, we can now extend some ideas of Liu, Lou, Peng and Zhou in \cite{LiuLouPengZho} to prove a monotonicity result with respect to the frequency $\omega$.
\begin{theorem} \label{theorem_monotonicity_omega_new}
Let $\tilde{\Omega}(s)$ be a $1$-periodic domain, and $\omega>0$. Let $\mu(\omega)$ be the principal periodic eigenvalue associated with $\Omega(t)=\tilde{\Omega}\left(\frac{\omega t}{2\pi}\right)$. Then $\mu(\omega)$ is monotonic non-decreasing: $\frac{d \mu(\omega)}{d\omega} \geq 0$. Moreover $\frac{d \mu(\omega)}{d\omega}=0$ if and only if the domain is independent of time.
\end{theorem}
\begin{proof}
Changing variables to $s = \frac{\omega t}{2\pi}$, the eigenfunctions $\psi_+(x,t)$ and $\overline{\psi}(x,t)$ from Lemma \ref{lemma_omega_even} now become functions $\phi_{\omega}(x,s)$ and $\overline{\phi}_{\omega}(x,s)$ which are positive on $x\in \tilde{\Omega}(s)$ and are $1$-periodic in $s$. By Lemma \ref{lemma_omega_even} they satisfy
\begin{equation}\label{eq_phi+}
\frac{\omega}{2\pi} \frac{\partial \phi_{\omega}}{\partial s} = D \nabla^2 \phi_{\omega} + \mu(\omega)\phi_{\omega} \qquad \textrm{for }x\in \tilde{\Omega}(s)
\end{equation}
\begin{equation}\label{eq_phi*}
-\frac{\omega}{2\pi} \frac{\partial \overline{\phi}_{\omega}}{\partial s} = D \nabla^2 \overline{\phi}_{\omega} + \mu(\omega)\overline{\phi}_{\omega} \qquad \textrm{for }x\in \tilde{\Omega}(s)
\end{equation}
\begin{equation}\label{eq_phi+*_BC}
\phi_{\omega}(x,s)=\overline{\phi}_{\omega}(x,s)=0 \qquad\textrm{for }x\in \partial \tilde{\Omega}(s) .
\end{equation}
Without loss of generality we may normalise them so that
\begin{equation}
\int_0^1\int_{\tilde{\Omega}(s)} \phi_{\omega}(x,s)^2 dxds = \int_0^1\int_{\tilde{\Omega}(s)} \phi_{\omega}(x,s)\overline{\phi}_{\omega}(x,s)dxds =1.
\end{equation}
Now, we follow the same steps as in \cite[Theorem 1.1]{LiuLouPengZho}. Namely we take equation \eqref{eq_phi+} and differentiate it with respect to $\omega$. Writing $\phi_{\omega}'$ for $\frac{ \partial \phi_{\omega}}{\partial \omega}$ and $\mu'(\omega)$ for $\frac{d\mu(\omega)}{d \omega}$, this becomes
\begin{equation}\label{eq_phi'}
\frac{1}{2\pi} \frac{\partial \phi_{\omega}}{\partial s} +
\frac{\omega}{2\pi} \frac{\partial \phi_{\omega}'}{\partial s} = D \nabla^2 \phi_{\omega}' + \mu(\omega)\phi_{\omega}' + \mu(\omega)\phi_{\omega}\qquad \textrm{for }x\in \tilde{\Omega}(s).
\end{equation}
Multiply this by $\overline{\phi}_{\omega}$ and integrate over $x \in \tilde{\Omega}(s)$ and $s \in [0,1]$. Using the boundary conditions \eqref{eq_phi+*_BC}, the equation \eqref{eq_phi*}, and the normalisation, this leads to
\begin{equation}\label{eq_omega'}
\frac{1}{2\pi} \int_0^1 \int_{\tilde{\Omega}(s)} \frac{\partial \phi_{\omega}}{\partial s} \overline{\phi}_{\omega} \ dx ds =  \mu'(\omega).
\end{equation}
Next, we define a functional
\begin{equation}
J_{\omega}(\zeta) = \int_0^1 \int_{\tilde{\Omega}(s)} \phi_{\omega}(x,s)\overline{\phi}_{\omega}(x,s) \left( \frac{ \frac{\omega}{2\pi} \frac{\partial \zeta}{\partial s}- D \nabla^2 \zeta}{\zeta(x,s)}\right) dxds
\end{equation}
for functions $\zeta(x,s)$ which are $C^2$ in $x\in\tilde{\Omega}(s)$ and $C^1$ on $\overline{\tilde{\Omega}(s)}$, and which are $C^1$ and $1$-periodic in $s\in[0,1]$, and which are positive for $x\in\tilde{\Omega}(s)$, with $\zeta(x,s) =0$ and $\nabla\zeta \cdot \nu \neq 0$ for $x\in\partial\tilde{\Omega}(s)$ (see \cite{LiuLouPengZho}).
Then it is straightforward to check using \eqref{eq_phi+} and \eqref{eq_phi*} that equation \eqref{eq_omega'} can be written as
\begin{equation}
\mu'(\omega) = \frac{1}{2\pi} \int_0^1 \int_{\tilde{\Omega}(s)} \frac{\partial \phi_{\omega}}{\partial s} \overline{\phi}_{\omega} dx ds = \frac{1}{2\omega}\left( J_{\omega}(\phi_{\omega}) -J_{\omega}(\overline{\phi}_{\omega}) \right).
\end{equation}
To prove the required result we must therefore show that $J_{\omega}(\phi_{\omega}) -J_{\omega}(\overline{\phi}_{\omega}) \geq 0$.
But this can be proved exactly as in the proof of Lemma 2.1 from  \cite{LiuLouPengZho}. Indeed, although our definition of $J_{\omega}$ now includes a time-dependent domain $\tilde{\Omega}(s)$, one can check that each step of the proof of \cite[Lemma 2.1, case 2 (b=1)]{LiuLouPengZho} (i.e. the case of Dirichlet boundary conditions) goes through exactly as in \cite{LiuLouPengZho}. This shows that, for all functions $\zeta$ of the class defined above,
\begin{equation}\label{eq_Jom_log}
J_{\omega}(\phi_{\omega}) -J_{\omega}(\zeta) = \int_0^1 \int_{\tilde{\Omega}(s)}
D \phi_{\omega}(x,s)\overline{\phi}_{\omega}(x,s) \left\vert \nabla\left(\log \frac{\zeta}{\phi_{\omega}}\right) \right\vert ^2 dxds,
\end{equation}
and the right hand side is clearly non-negative.
In particular, $J_{\omega}(\phi_{\omega}) -J_{\omega}(\overline{\phi}_{\omega}) \geq 0$.
This proves that $\mu'(\omega)\geq0$. Moreover, $\mu'(\omega)=0$ if and only if $J_{\omega}(\phi_{\omega}) -J_{\omega}(\overline{\phi}_{\omega}) = 0$, and by equation \eqref{eq_Jom_log} this holds if and only if $\frac{\overline{\phi}_{\omega}}{\phi_{\omega}}$ is a function just of $s$.
Substituting $\overline{\phi}_{\omega}(x,s)=\beta(s)\phi_{\omega}(x,s)$ into \eqref{eq_phi*} and using \eqref{eq_phi+} gives
\begin{equation}
\beta'(s)\phi_{\omega}+2\beta(s)\frac{\partial \phi_{\omega}}{\partial s}=0.
\end{equation}
Since $\phi_{\omega}=0$ on $\partial\tilde{\Omega}(s)$ and $\beta(s)>0$, this implies that also the (possibly one-sided) derivative $\frac{\partial \phi_{\omega}}{\partial s}$ is zero at each point $x$ on $\partial\tilde{\Omega}(s)$. As $\phi_{\omega}>0$ on the interior and is zero  with non-zero normal derivative on the boundary, we deduce that in fact the boundary $\partial\tilde{\Omega}(s)$ must remain the same for all times $s$: the domain is time-independent. Conversely, for all such time-independent domains, we have $\beta=1$ and $\mu$ does not depend on $\omega$. Thus $\mu'(\omega)=0$ if and only if the domain is independent of time.
\end{proof}
\begin{remark}
It is possible to use the same proof to extend Lemma \ref{lemma_omega_even} to operators of the form $\frac{\partial}{\partial t}-D\nabla^2+V(t)$ on $\Omega(t)$, where $V(t)$ (independent of $x$) is a continuous and periodic function with the same period as the domain, and satisfies the extra condition that $V(t)\equiv V(-t)$. For operators of this form, the monotonicity result of Theorem \ref{theorem_monotonicity_omega_new} then also follows as above. However, the monotonicity results may not carry over to more general forms of the operator.
\end{remark}

\section{Nonlinear equation on a periodic domain}\label{section_nonlinear}
In this section, we consider the nonlinear periodic parabolic problem \eqref{eq_u_ij}, \eqref{eq_u_ij_BC} where $f$ is assumed to satisfy the conditions \eqref{assumptions_on_f}.
As above, let $\mu$ and $\phi(\xi,t)$ be the principal periodic eigenvalue and eigenfunction satisfying \eqref{eq_princ_eig1}, \eqref{eq_princ_eig2}, \eqref{eq_princ_eig3}, \eqref{eq_princ_eig4}, and normalised so that $\vert\vert \phi \vert\vert_{\infty}=1$. Now the solution to the linear equation is a supersolution to the nonlinear problem, so if $f'(0)<\mu$ then $u\rightarrow 0$ as $t\rightarrow \infty$.

From now on, assume $f'(0)>\mu$. Fix any $\alpha \in (0, f'(0)-\mu)$. Then since $f(u)=f'(0)u+o(u)$ as $u\rightarrow0$ there exists $\varepsilon>0$ (depending on $\alpha$) such that for all $0\leq u\leq \varepsilon$,
\begin{equation}
(\alpha - f'(0) +\mu)u +(f'(0)u-f(u))\leq 0.
\end{equation}
Now, for any $\delta$ such that $0<\delta\leq\varepsilon e^{-\alpha T}$, the function $\hat{u}(\xi,t)=\delta \phi(\xi,t)e^{\alpha t}$ is a subsolution for $u$ for $0\leq t\leq T$:
\begin{align}
\frac{\partial \hat{u}}{\partial t} - \mathcal{L}\hat{u} - f(\hat{u}) &= \alpha\hat{u} + \mu\hat{u}- f(\hat{u}) \\
&= (\alpha - f'(0) +\mu)\hat{u} +(f'(0)\hat{u}-f(\hat{u})) \ \leq 0
\end{align}
since $\hat{u}(\xi,t)\leq \varepsilon$ for $0\leq t\leq T$. The function $\hat{u}$ also satisfies $\hat{u}(\xi,t)=0$ on $\partial\Omega_0$, and $\hat{u}(\xi,0)\leq\hat{u}(\xi,T)$, and so it is a subsolution to the periodic problem \eqref{eq_u_ij}, \eqref{eq_u_ij_BC} in the sense of Hess \cite[chapter III Definition 21.1]{Hess}. Moreover the constant $K$ is a supersolution. By applying \cite[Theorem 22.3, chapter III]{Hess}, there exists a stable periodic solution $u^{*}(\xi,t)$ to
\begin{equation}\label{eq_periodic_sol1}
\frac{\partial u^{*}}{\partial t} = \mathcal{L}u^{*} +f(u^{*})  \qquad \textrm{in } \xi\in\Omega_0, \ t\in\mathbb{R}
\end{equation}
\begin{equation}\label{eq_periodic_sol2}
u^{*}(\xi,t)=0 \qquad \textrm\qquad\textrm{for } \xi\in \partial\Omega_0
\end{equation}
\begin{equation}\label{eq_periodic_sol3}
u^{*}(\xi,t)\equiv u^{*}(\xi,t+T)
\end{equation}
such that
\begin{equation}
\varepsilon\phi(\xi,t)e^{\alpha (t-T)} \leq u^{*}(\xi,t) \leq K \qquad \textrm{for } 0\leq t\leq T,\ \xi\in \Omega_0.
\end{equation}
In the remainder of this section, we shall prove that the periodic solution $u^{*}$ is unique and that for any initial conditions $0\leq u(\xi,0)\leq K$ not identically zero, the solution $u(\xi,t)$ to the problem \eqref{eq_u_ij}, \eqref{eq_u_ij_BC} converges to $u^{*}$.
\begin{remark}
It is straightforward to derive a lower bound on $u(\xi,t)$ which shows that $u(\xi,t)$ cannot converge to zero. By the strong maximum principle and Hopf's lemma, we can assume without loss of generality that there exists $0<\delta \leq\varepsilon e^{-\alpha T}$ such that $\delta\phi(\xi,0) \leq u(\xi,0)$ (since this will hold for every time $t_0>0$). Then, since $\delta\phi(\xi,t)e^{\alpha t}$ is a subsolution on $0\leq t\leq T$,
\begin{equation}
\delta\phi(\xi,t')e^{\alpha t'} \leq u(\xi,t') \qquad \textrm{for all }
0\leq t'\leq T.
\end{equation}
Then as a consequence of the $T$-periodicity of $\phi$, $\delta\phi(\xi,0)\leq \delta\phi(\xi,0)e^{\alpha T} \leq u(\xi,T)$ and we can conclude that
\begin{equation}
\delta\phi(\xi,t')e^{\alpha t'} \leq u(\xi,t'+nT) \qquad \textrm{for all }0\leq t'\leq T,\ n\in\mathbb{N}.
\end{equation}
Therefore
\begin{equation}
\liminf_{t\rightarrow\infty} u(\xi,t) \geq \delta\min_{0\leq t' \leq T}(\phi(\xi,t')e^{\alpha t'}).
\end{equation}
\end{remark}
To prove the convergence to $u^{*}$, we shall use the Poincar\'{e} map $P_T$. For each $\tau> 0$ define $P_{\tau}$ to be the map  $P_{\tau}(u_0)=u(\cdot,{\tau})$
where $u$ is the solution to the problem \eqref{eq_u_ij}, \eqref{eq_u_ij_BC} with initial conditions $u(\cdot,0)=u_0(\cdot)$. Since the coefficients are periodic, this is the same as the map taking $u(\cdot, nT)$ to $u(\cdot, nT+\tau)$ for any $n\in \mathbb{N}$. The Poincar\'{e} map is $P_T$, which takes the solution at time $nT$ to the solution at time $(n+1)T$. If $u^{*}$ is a $T$-periodic solution (satisfying \eqref{eq_periodic_sol1}, \eqref{eq_periodic_sol2}, \eqref{eq_periodic_sol3}), then $u^{*}(\cdot, 0)$ is a fixed point of the Poincar\'{e} map $P_T$.

We shall use the following two properties of $P_{\tau}$.
\begin{lemma}\label{lemma_monotonicity} Monotonicity of $P_{\tau}$.\\
For any $\tau>0$, the map $P_{\tau}$ is monotonic: if $u_0 \leq v_0$ then $P_{\tau}(u_0) \leq P_{\tau}(v_0)$. Moreover, either $u_0 \equiv v_0$ or else there is strict inequality $P_{\tau}(u_0) < P_{\tau}(v_0)$ in $\Omega_0$ and the normal derivatives satisfy $\frac{\partial}{\partial \nu}P_{\tau}(u_0) \neq \frac{\partial}{\partial \nu}P_{\tau}(v_0)$ on $\partial\Omega_0$.
\end{lemma}
\begin{proof}
This is a consequence of the parabolic comparison principle, strong maximum principle, and Hopf's Lemma.
\end{proof}
\begin{lemma} \label{lemma_sublinearity} Sublinearity of $P_{\tau}$.\\
Let $f$ satisfy \eqref{assumptions_on_f}. Then for any $\tau>0$, the map $P_{\tau}$ is sublinear, in the following sense. Let $0\leq\alpha\leq 1$, and let $u_0>0$ on $\Omega_0$, with $u_0=0$ and $\frac{\partial u_0}{\partial \nu}\neq 0$ on $\partial\Omega_0$. Then
\begin{equation}
\alpha P_{\tau}(u_0) \leq P_{\tau}(\alpha u_0).
\end{equation}
\end{lemma}
\begin{proof}
If $\alpha=0$ or $1$ then it is obvious, so assume $0<\alpha<1$. Let $u(\xi,t)$ be the solution to \eqref{eq_u_ij}, \eqref{eq_u_ij_BC} with initial conditions $u(\xi,0)=u_0(\xi)$ and $v(\xi,t)$ the solution with initial conditions $v(\xi,0)=\alpha u_0(\xi)$. We need to show that $\alpha u(\xi,t) \leq v(\xi,t)$ for all $t\geq 0$.

By the assumption that $\frac{f(k)}{k}$ is non-increasing, we have $f(\alpha u_0)\geq\alpha f(u_0)$. For $\varepsilon>0$ small, define $f_{\varepsilon}(k)=f(k)-\varepsilon k^2$, so that $\frac{f_{\varepsilon}(k)}{k}$ is strictly decreasing in $k>0$, and
\begin{equation}
f_{\varepsilon}(\alpha u_0)- \alpha f_{\varepsilon}(u_0) \geq \varepsilon\alpha(1-\alpha)u_0^2 \ >0\quad\textrm{in }\Omega_0.
\end{equation}
Let $v_{\varepsilon}$, $u_{\varepsilon}$ be the corresponding solutions to the problem with $f$ replaced by $f_{\varepsilon}$:
\begin{equation}
\frac{\partial u_{\varepsilon}}{\partial t}=\mathcal{L}u_{\varepsilon} +f_{\varepsilon}(u_{\varepsilon}), \qquad
\frac{\partial v_{\varepsilon}}{\partial t}=\mathcal{L}v_{\varepsilon} +f_{\varepsilon}(v_{\varepsilon}).
\end{equation}
We shall show that $\alpha u_{\varepsilon}(\xi,t) \leq v_{\varepsilon}(\xi,t)$ for every $t\geq 0$. Then by taking $\varepsilon\rightarrow 0$ we conclude that the same inequality holds for the solutions $v$, $u$ with the original reaction function $f$.

At $t=0$ we have $v_{\varepsilon}(\cdot,0)- \alpha u_{\varepsilon}(\cdot,0)=0$ and
\begin{equation}
\frac{\partial}{\partial t}(v_{\varepsilon}-\alpha u_{\varepsilon})\vert_{t=0} = \mathcal{L}(\alpha u_0)+f_{\varepsilon}(\alpha u_0) - \alpha\mathcal{L}(u_0)-\alpha f_{\varepsilon}(u_0) =f_{\varepsilon}(\alpha u_0) -\alpha f_{\varepsilon}(u_0)  \geq \varepsilon\alpha(1-\alpha)u_0^2.
\end{equation}
Therefore, there exists $t_0>0$ such that
$0\leq v_{\varepsilon}(\xi,t)- \alpha u_{\varepsilon}(\xi,t)$ for $0\leq t\leq t_0$. If $t_0$ can be taken as large as we like, then we are done. Otherwise, let $t^{*}$ be the maximal such that $\alpha u_{\varepsilon}(\xi,t) \leq v_{\varepsilon}(\xi,t)$ for $0\leq t\leq t^{*}$.
Let $\tilde{v}_{\varepsilon}$ be the solution on $t\geq t^{*}$ with $\tilde{v_{\varepsilon}}(\xi,t^{*})=\alpha u_{\varepsilon}(\xi,t^{*})$. Then by applying the same argument as above, to the function $\tilde{v}_{\varepsilon}$ at time $t^{*}$, we deduce that
there exists $t_1>0$ such that $\alpha u_{\varepsilon} \leq \tilde{v}_{\varepsilon}$ for $t^{*}\leq t\leq t^{*}+t_1$. Since $\tilde{v}_{\varepsilon}\leq v_{\varepsilon}$, this contradicts the maximality of $t^{*}$. Therefore, we do have $\alpha u_{\varepsilon}(\xi,t) \leq v_{\varepsilon}(\xi,t)$ for all $t\geq 0$, as required.
\end{proof} 

\begin{theorem}Uniqueness of periodic solution (given ordering). \label{theorem_unique_per}\\
Suppose $f'(0)>\mu$, and suppose that $\underline{U}(\xi,t)$, $\overline{U}(\xi,t)$ are both positive, $T$-periodic solutions to the problem \eqref{eq_periodic_sol1}, \eqref{eq_periodic_sol2}, \eqref{eq_periodic_sol3}
with $0\leq \underline{U}(\xi,0)\leq\overline{U}(\xi,0)$ for all $\xi\in\Omega_0$. Then $\underline{U}(\xi,t)\equiv \overline{U}(\xi,t)$.
\end{theorem}
\begin{proof}
By the strong maximum principle and Hopf's lemma, $0<\underline{U}(\xi,0)\leq\overline{U}(\xi,0)$ for all $\xi\in\Omega_0$, and $\underline{U}$ and $\overline{U}$ have non-zero normal derivatives on $\partial\Omega_0$. Therefore for $r>0$ small enough we have $r\overline{U}(\xi,0)\leq \underline{U}(\xi,0)$ for all $\xi\in\Omega_0$. On the other hand this does not hold for any $r>1$. Let 
\begin{equation}
\hat{r}=\sup\{ r\in(0,1): r\overline{U}(\xi,0)\leq \underline{U}(\xi,0) \textrm{ for all } \xi\in\Omega_0 \}.
\end{equation}
Then we know that
\begin{equation}
\hat{r}\overline{U}(\xi,0)\leq \underline{U}(\xi,0) \textrm{ for all } \xi\in\Omega_0
\end{equation}
and (by maximality of $\hat{r}$) there exists some
\begin{equation}\label{eq_xi0}
\xi_0\in\Omega_0 \textrm{ such that }\hat{r}\overline{U}(\xi_0,0) = \underline{U}(\xi_0,0) \qquad \textrm{or}\qquad
\xi_0 \in \partial\Omega_0 \textrm{ such that } \hat{r}\frac{\partial \overline{U}}{\partial \nu}(\xi_0,0) =  \frac{\partial \underline{U}}{\partial \nu}(\xi_0,0).
\end{equation}
Now apply the Poincar\'{e} map, $P_T$. By the monotonicity (Lemma \ref{lemma_monotonicity}) we have
\begin{equation}
P_T(\hat{r}\overline{U}(\cdot,0)) \leq P_T(\underline{U}(\cdot,0)) 
\end{equation}
with either $\hat{r}\overline{U} \equiv \underline{U}$ or else strict inequality 
\begin{equation}\label{eq_strict1}
P_T(\hat{r}\overline{U}(\cdot,0)) < P_T(\underline{U}(\cdot,0)) \qquad \textrm{on } \Omega_0
\end{equation}
and
\begin{equation}\label{eq_strict2}
\frac{\partial}{\partial \nu}P_T(\hat{r}\overline{U}(\cdot,0)) \neq \frac{\partial}{\partial \nu}P_T(\underline{U}(\cdot,0)) \qquad \textrm{on } \partial\Omega_0.
\end{equation}
Combining this with the sublinearity property (Lemma \ref{lemma_sublinearity}) and the fact that $\overline{U}$ and $\underline{U}$ are fixed points of $P_T$, we find that
\begin{equation}\label{ineq_for_uniqueness}
\hat{r}\overline{U}(\cdot,0) = \hat{r}P_T(\overline{U}(\cdot,0)) \leq 
P_T(\hat{r}\overline{U}(\cdot,0)) \leq P_T(\underline{U}(\cdot,0))=\underline{U}(\cdot,0)
\end{equation}
and that either $\hat{r}\overline{U} \equiv \underline{U}$ or else equations \eqref{eq_strict1} and \eqref{eq_strict2} hold. Incorporating these strict inequalities into equation \eqref{ineq_for_uniqueness} would contradict the existence of $\xi_0$ as in equation \eqref{eq_xi0}. Therefore, in fact
\begin{equation}
\hat{r}\overline{U} \equiv \underline{U} \qquad\textrm{on } \overline{\Omega_0}\times[0,T].
\end{equation}
This shows that $\overline{U}$ and $\hat{r}\overline{U}$ are both solutions to \eqref{eq_periodic_sol1}, \eqref{eq_periodic_sol2}, \eqref{eq_periodic_sol3}, and hence $\hat{r}f(\overline{U})\equiv f(\hat{r}\overline{U})$. By the assumption that $\frac{f(k)}{k}$ is non-increasing on $k>0$, this implies that either $\hat{r}=1$ or else $f(\overline{U}) \equiv f'(0)\overline{U}$.
But we know that $\overline{U}$ does not satisfy the linear equation because that would contradict the fact that $f'(0)>\mu$. Therefore, it must be that $\hat{r}=1$ and $\underline{U}\equiv \overline{U}$.
\end{proof}
Now we are able to prove convergence to $u^{*}(\xi,t)$ (the positive $T$-periodic solution to \eqref{eq_periodic_sol1}, \eqref{eq_periodic_sol2}, \eqref{eq_periodic_sol3} whose existence is guaranteed by \cite[Theorem 22.3, chapter III]{Hess}).
\begin{theorem}\label{theorem_convergence1}
Assume that $f$ satisfies \eqref{assumptions_on_f} and $f'(0)>\mu$, and let $u^{*}(\xi,t)$ be a positive $T$-periodic solution to \eqref{eq_periodic_sol1}, \eqref{eq_periodic_sol2}, \eqref{eq_periodic_sol3}.
Given non-negative, not identically zero initial conditions $0\leq u(\xi,0)\leq K$, let $u(\xi,t)$ be the solution to the nonlinear problem \eqref{eq_u_ij}, \eqref{eq_u_ij_BC}, and for $n\in\mathbb{N}$ define $u_n(\xi, t)=u(\xi,nT+t)$. Then as $n\rightarrow\infty$, $u_n$ converges in $C^{2,1}(\overline{\Omega_0}\times[0,T])$ to $u^{*}(\xi,t)$. In particular, $u^{*}$ is unique.
\end{theorem}
\begin{proof}
Without loss of generality (since it will hold for every time $t_0>0$ by the strong maximum principle and Hopf's Lemma) we can assume that the initial conditions are such that
\begin{equation}
\delta u^{*}(\xi,0) \leq u(\xi, 0)\leq Bu^{*}(\xi,0)
\end{equation}
for some $0<\delta \leq 1$ and $B\geq1$.
Let $\underline{u}(\xi,t)$ and $\overline{u}(\xi,t)$ be the solutions to \eqref{eq_u_ij}, \eqref{eq_u_ij_BC} with initial conditions $\underline{u}(\xi,0)=\delta u^{*}(\xi,0)$ and $\overline{u}(\xi,0)=B u^{*}(\xi,0)$. By the comparison principle,
\begin{equation}\label{eq_u_ineq}
\underline{u}(\cdot,t)\leq u(\cdot,t)\leq \overline{u}(\cdot,t) \qquad \textrm{and}\qquad \underline{u}(\cdot,t)\leq u^{*}(\cdot,t)\leq \overline{u}(\cdot,t)\qquad \textrm{for all } t\geq0.
\end{equation}
For $n\in\mathbb{N}$ define $u_n(\xi, t)=u(\xi,nT+t)$; also define $\underline{u}_n(\xi,t)=\underline{u}(\xi,nT+t)$ and $\overline{u}_n(\xi,t)=\overline{u}(\xi,nT+t)$. Then we have
\begin{equation}\label{eq_u_n_ineq}
\underline{u}_n(\cdot,t)\leq u_n(\cdot,t)\leq \overline{u}_n(\cdot,t)\qquad \textrm{and}\qquad \underline{u}_n(\cdot,t)\leq u^{*}(\cdot,t)\leq \overline{u}_n(\cdot,t)
\end{equation}
for all $0\leq t\leq T$, $n\in\mathbb{N}$.

Using the fact that $u^{*}(\cdot, 0)$ is a fixed point of the Poincar\'{e} map $P_T$, together with the sublinearity of $P_T$ (Lemma \ref{lemma_sublinearity}), we get that for all $\xi\in\Omega_0$,
\begin{equation}
\underline{u}(\xi,0)=\delta u^{*}(\xi,0) = \delta P_T(u^{*}(\xi,0)) \leq
P_T(\delta u^{*}(\xi,0))= P_T(\underline{u}(\xi,0))=\underline{u}(\xi,T)
\end{equation}
and
\begin{align}
\overline{u}(\xi,0)= Bu^{*}(\xi,0) =B P_T(u^{*}(\xi,0)) =
BP_T\left(\frac{1}{B} \ Bu^{*}(\xi,0)\right) \geq P_T( Bu^{*}(\xi,0))= P_T(\overline{u}(\xi,0))= \overline{u}(\xi,T).
\end{align}
Therefore, $\underline{u}(\xi,0) \leq \underline{u}(\xi,T)$ and  $\overline{u}(\xi,T) \leq \overline{u}(\xi,0)$. Apply $P_T$ again and use the monotonicity property (Lemma \ref{lemma_monotonicity}) and the ordering \eqref{eq_u_ineq}, to deduce that
\begin{equation}
\underline{u}(\xi, nT)\leq \underline{u}(\xi, (n+1)T) \leq u^{*}(\xi, 0)\leq \overline{u}(\xi, (n+1)T)\leq \overline{u}(\xi, nT)
\end{equation}
for all $\xi\in\Omega_0$, $n\in\mathbb{N}$.
Therefore, pointwise limits $\underline{v}(\xi)\leq\overline{v}(\xi)$ exist such that $\underline{v}(\xi) \leq u^{*}(\xi, 0)\leq \overline{v}(\xi)$ and
\begin{equation}
\underline{u}(\xi, nT)\rightarrow \underline{v}(\xi), \qquad  \overline{u}(\xi,nT)\rightarrow \overline{v}(\xi) \qquad \textrm{as }n\rightarrow\infty.
\end{equation}
Using parabolic estimates from \cite[Lemma 7.20 and Theorem 7.30]{Lieberman} and \cite[chapter IV, Theorem 10.1]{Ladyzhenskaya} and embeddings from \cite[Theorem 3.14(3)]{BeiHu}, we deduce that there is a subsequence $\underline{u}_{n_k}$ which converges in $C^{2,1}(\overline{\Omega_0}\times[0,T])$ to a solution $\underline{U}(\xi,t)$ of the nonlinear parabolic problem \eqref{eq_u_ij}, \eqref{eq_u_ij_BC}. By equating this to the pointwise limit at times $0$ and $T$, we have that $\underline{U}(\xi,0)=\underline{U}(\xi,T)=\underline{v}(\xi)$.
Likewise, there is a subsequence $\overline{u}_{n_r}$ of $\overline{u}_n$ which converges in $C^{2,1}(\overline{\Omega_0}\times[0,T])$ to a solution $\overline{U}(\xi,t)$ of \eqref{eq_u_ij}, \eqref{eq_u_ij_BC}, with $\overline{U}(\xi,0)=\overline{U}(\xi,T)=\overline{v}(\xi)$.

Now $\underline{U}(\xi,0)=\underline{v}(\xi)\leq u^{*}(\xi, 0)\leq \overline{v}(\xi)=\overline{U}(\xi,0)$ and so by the comparison principle, $\underline{U}(\xi,t)\leq u^{*}(\xi, t)\leq \overline{U}(\xi,t)$ for all $t\geq 0$. Therefore $\underline{U}$ and $\overline{U}$ satisfy the conditions of Theorem \ref{theorem_unique_per}, and we conclude that
\begin{equation}
\underline{U}\equiv \overline{U} \equiv u^{*}.
\end{equation}
Since the limit is uniquely identified, this implies that actually the whole sequences $\underline{u}_n$ and $\overline{u}_n$ converge to $u^{*}$ as $n\rightarrow\infty$ and the convergence is in $C^{2,1}(\overline{\Omega_0}\times[0,T])$.
But since $u_n$ satisfies \eqref{eq_u_n_ineq}, it must also converge uniformly to $u^{*}$ as $n\rightarrow \infty$, and by the same argument as above the convergence is in $C^{2,1}(\overline{\Omega_0}\times[0,T])$.
\end{proof}
The convergence of $u(\xi,nT+t)$ to a unique positive $T$-periodic solution $u^{*}(\xi,t)$ on $\Omega_0\times[0,T]$ can now be interpreted in terms of the original problem for $\psi(x,t)$ on the $T$-periodic domain $\Omega(t)$. The function $u^{*}(\xi,t)$ for $\xi\in\Omega_0$ corresponds to a positive solution $\psi^{*}(x,t)$ to \eqref{eq_psi}, \eqref{eq_psi_BC} such that $\psi^{*}(x,t)\equiv \psi^{*}(x,t+T)$ for all $x\in\Omega(t)$, $t\in\mathbb{R}$. Theorem \ref{theorem_convergence1} means that $\psi(x,nT+t)$ converges uniformly to $\psi^{*}(x,t)$ as $n\rightarrow\infty$.

\section*{Acknowledgements}
Much of this work was carried out during my PhD at Swansea University, and I would like to thank my supervisor Professor Elaine Crooks. I am grateful for an
EPSRC-funded studentship: EPSRC DTP grant EP/R51312X/1,
and research associate funding: EP/W522545/1.

\end{document}